\documentclass[11pt]{article}

\usepackage{graphicx}
\usepackage{amsfonts}
\usepackage{amscd}
\usepackage{indentfirst}
\usepackage{amssymb, amsmath,amsthm, amsgen, amstext, amsbsy, amsopn,color}

\newtheorem{thm}{Theorem}[section]
\newtheorem{lemma}[thm]{Lemma}
\newtheorem{prop}[thm]{Proposition}
\newtheorem{cor}[thm]{Corollary}

\newtheorem{rem}[thm]{Remark}

\newtheorem{question}[thm]{Question}

\newcommand{\R}{{\mathbb{R}}}

\newcommand{\T}{{\mathbb{T}}}
\newcommand{\Z}{{\mathbb{Z}}}
\newcommand{\N}{{\mathbb{N}}}
\newcommand{\C}{{\mathbb{C}}}
\newcommand{\SP}{{\mathbb{S}}}

\newcommand{\D}{{\mathbb{D}}}

\newcommand{\im}{{\it Im\, }}
\newcommand{\gobp}{{\mathfrak{bp}}}

\newcommand{\cD}{{\mathcal{D}}}

\newcommand{\cF}{{\mathcal{F}}}

\newcommand{\cI}{{\mathcal{I}}}
\newcommand{\cJ}{{\mathcal{J}}}
\newcommand{\cK}{{\mathcal{K}}}

\newcommand{\cM}{{\mathcal{M}}}
\newcommand{\cN}{{\mathcal{N}}}

\newcommand{\hF}{{\widehat{F}}}
\newcommand{\hG}{{\widehat{G}}}

\newcommand{\hJ}{{\widehat{J}}}
\newcommand{\hK}{{\widehat{K}}}
\newcommand{\hL}{{\widehat{L}}}
\newcommand{\hM}{{\widehat{M}}}

\newcommand{\hX}{{\widehat{X}}}
\newcommand{\hY}{{\widehat{Y}}}

\newcommand{\hp}{{\hat{p}}}

\newcommand{\bx}{{\bf x}}
\newcommand{\by}{{\bf y}}
\newcommand{\bp}{{\it bp}}

\newcommand{\halpha}{{\hat{\alpha}}}
\newcommand{\homega}{{\hat{\omega}}}

\newcommand{\hgamma}{{\hat{\gamma}}}

\newcommand{\hGamma}{{\widehat{\Gamma}}}

\newcommand{\goS}{{\mathfrak{S}}}

\newcommand{\tcM}{{\widetilde{\cM}}}

\newcommand{\df}{{\hbox{\it def}\,}}
\newcommand{\Flux}{{\hbox{\it Flux}}}
\newcommand{\modd}{{\footnotesize\,\rm mod}\,}

\newcommand{\Int}{\hbox{\rm Int}\, }
\newcommand{\Span}{\hbox{\rm Span}}
\newcommand{\Ker}{\hbox{\rm Ker}\,}
\newcommand{\Image}{\hbox{\rm Im}\,}
\newcommand{\Lie}{\hbox{\it Lie}\, }
\newcommand{\Lied}{\textit{Lie}^*\, }
\newcommand{\Diff}{\hbox{\rm Diff}}

\newcommand{\Qed}{\ \hfill \qedsymbol \bigskip}

\DeclareMathOperator{\sgrad}{sgrad}

\begin{document}

\title{Lagrangian isotopies and symplectic function theory }

\author{\textsc Michael Entov$^{1}$,\ Yaniv Ganor$^{2}$,\ Cedric Membrez$^{3}$}

\footnotetext[1]{Partially supported by the Israel Science Foundation grant
$\#$ 1096/14.}
\footnotetext[2]{Partially supported by European Research Council Advanced grant 338809.}
\footnotetext[3]{Partially supported by the Swiss National Science Foundation
grant 155540 and the European Research Council Advanced grant 338809.}

\date{\today}

\maketitle

\begin{abstract}

We study two related invariants of Lagrangian submanifolds in symplectic
manifolds. For a Lagrangian torus these invariants are functions on the
first cohomology of the torus.

The first invariant is of topological nature and is related to the study of
Lagrangian isotopies with a given Lagrangian flux. More specifically, it
measures the length of straight paths in the first cohomology that can be
realized as the Lagrangian flux of a Lagrangian isotopy.

The second invariant is of analytical nature and comes from symplectic
function theory. It is defined
for Lagrangian submanifolds admitting fibrations over a circle and has a
dynamical interpretation.

We partially compute these invariants for certain Lagrangian tori.

\end{abstract}

\newpage
\tableofcontents
\newpage


\section{Introduction}
\label{sec-intro}

In this paper we study two related invariants of Lagrangian submanifolds
which are invariant under symplectomorphisms of the ambient symplectic
manifolds.

Let $(M,\omega)$ be a symplectic manifold (possibly with boundary)
and $L\subset (M,\omega)$ a
closed Lagrangian submanifold. Let $\iota: L\to M$ denote here and further
on the natural inclusion of $L$ in $M$.

The first invariant of $L$ comes from the consideration of Lagrangian
isotopies of $L$ with a given Lagrangian flux path in $H^1(L; \R)$.
Namely, recall that if $\psi:=\{ \psi_t: L\to M\}_{0\leq t\leq T}$,
$\psi_0=\iota$, is a Lagrangian isotopy, one can associate to it a
{\it Lagrangian flux path}
\[
\{\Flux(\psi)_t\}_{0\leq t\leq T} \subset H^1 (L;\R), \ \ \Flux(\psi)_0 = 0,
\]
as follows: given a closed curve $C \subset L$ and $t\in [0,T]$, consider
the trace of $C$ under the Lagrangian isotopy $\{ \psi_\tau: L\to M\}_{
0 \leq \tau \leq t}$ and integrate $\omega$ over the resulting surface.
The resulting (real) number depends only on the homology class of $C$
in $H_1 (L)$ (see \cite{Solomon}). The numbers obtained for all $C$ in
this way are the periods of a uniquely defined {\it Lagrangian flux}
class $\Flux(\psi)_t \in H^1 (L;\R)$.

The notion of Lagrangian flux immediately raises the following question.

\begin{question}
Which paths in $H^1 (L;\R)$ based at 0 have the form $\{ \Flux(\psi)_t\}$
for some Lagrangian isotopy $\psi$ of $L$?
\end{question}

In general, this seems to be a very difficult question. In this paper
we investigate its weaker version:

\begin{question}
Which {\bf straight} paths in $H^1 (L;\R)$ based at 0 have the
form $\{ \Flux(\psi)_t\}$ for some Lagrangian isotopy $\psi$ of $L$?
\end{question}

The {\it deformation invariant} of $L$ is a function $\df_L: H^1 (L;\R)
\to (0,+\infty]$ defined as follows. Given $\alpha\in H^1 (L;\R)$, set
\[
\df_L (\alpha):= \sup T,
\]
where we take the supremum over all $T \in \R_{\geq 0}$ for which there
exists a Lagrangian isotopy $\psi = \{ \psi_t: L\to M\}_{0 \leq t \leq T}$,
$\psi_0=\iota$, of $L$ such that
\begin{equation}
\label{eqn-flux-path-linear}
\Flux(\psi)_t = -t \alpha\ \textrm{ for all }\ 0\leq t\leq T,
\end{equation}
or, in other words, the path $\{ - t \alpha\}_{0\leq t\leq T}$ in
$H^1 (L;\R)$ is the path $\{ \Flux(\psi)_t\}$ for some Lagrangian isotopy
$\psi$ of $L$. It is easy to show that $\df_L (\alpha)$ is always
non-zero -- thus, $\df_L (\alpha)$ takes values in $(0,+\infty]$.

In case $(M,\omega=d\lambda)$ is an exact symplectic manifold the invariant
$\df_L$ can be related to the notion of a symplectic shape studied in
\cite{Eliashberg, Sikorav2}. The symplectic shape of $(M,d\lambda)$ associated
to $L$ and a homomorphism $h: H^1 (M;\R)\to H^1 (L;\R)$ is the subset of $H^1 (L;\R)$ formed by the Liouville
classes $[e^*\lambda] \in H^1(L; \R)$ for all possible Lagrangian embeddings
$e: L \to (M,d\lambda)$ such that $h=e^*: H^1 (M;\R)\to H^1 (L;\R)$.

Assume $M=\T^n\times U\subset \T^n\times \R^n
= T^* \T^n$, where $U\subset \R^n$ is a connected open set, and $\lambda$ is the
standard Liouville form on $\T^n\times \R^n$.
Let $\bx \in U$ and let $L := \T^n \times \{\bx\}\subset \T^n\times U$
be the corresponding Lagrangian submanifold.
Let $h: H^1 (\T^n\times U;\R)\to H^1 (\T^n\times \{\bx\};\R)$ be induced by the embedding
$\T^n\times \{ \bx\} \hookrightarrow \T^n\times U$. Then the Benci-Sikorav theorem \cite{Sikorav2}
(cf. \cite{Eliashberg}) says that the shape of $(\T^n\times U,d\lambda)$ associated to
$L$ and $h$ is $U\subset \R^n\cong H^1 (\T^n\times \{\bx\};\R)$. Rephrasing this result in terms
of the deformation invariant gives us the following theorem.
\begin{thm}
\label{thm-def-shapes}
With $U$ and $(M=\T^n\times U, d\lambda)$ and $L = \T^n \times \{\bx\}$ as above,
for all $\alpha \in H^1(L;\R)$
\[
\df_L (\alpha) = \sup \{ \ t > 0 \ | \ \bx - t \alpha \in U \ \}.
\]
\end{thm}

To prove Theorem~\ref{thm-def-shapes} note that in this setting
$\Flux(\psi)_t$ can be represented as the difference of the Liouville classes of
$\psi_t (L)$ and $L$:
\[
\Flux(\psi)_t = \psi_t^* \lambda - \psi_0^*\lambda.
\]
Then the inequality $\df_L (\alpha) \geq \sup \{ \ t > 0 \ | \ \bx - t \alpha \in U \ \}$
follows from the existence of an obvious Lagrangian isotopy $\psi_t (\T^n \times \{\bx\})
= \T^n\times \{ \bx-t\alpha\}$, while the opposite inequality follows directly from the
Benci-Sikorav theorem.

The function $1/\df_L$ is obviously non-negatively homogeneous.
In the case when $U$ is star-shaped with respect to $\bx$,  it
is, by Theorem~\ref{thm-def-shapes}, the Min\-kow\-ski functional associated
to $U$ -- which in turn completely determines $U$.

The study of $\df_L$ can also be viewed as a relative analogue of a deformation
problem for symplectic forms on closed symplectic manifolds considered in
\cite{Polterovich}:
how far can one deform a symplectic form $\omega$ within a family of symplectic
forms such that the cohomology class of the deformed form changes
along a straight ray originating at $[\omega]$ and such that its restriction
to a given $\omega$-symplectic submanifold remains symplectic?
To see how a relative version of this question is related to $\df_L$
we denote by $[\omega]_L\in H^2 (M,L;\R)$ and $[\omega]_t^\psi\in H^2 (M,\psi_t (L);\R)$
the relative symplectic area cohomology classes of, respectively,
$L$ and $\psi_t (L)$. The isotopy $\psi$ defines a canonical isomorphism
$H^2 (M,L;\R)\cong H^2 (M,\psi_t (L);\R)$
and thus $[\omega]_t^\psi$ can be viewed as an element of $H^2 (M,L;\R)$.
Let $\partial: H^1(L; \R) \to H^2(M, L; \R)$ be the connecting homomorphism.
Then
\[
[\omega]^{\psi}_t = [\omega]_L + \partial \Flux(\psi)_t
\]
and condition \eqref{eqn-flux-path-linear} becomes
\begin{equation}
\label{eqn-omega-psi-t-linear}
[\omega]^{\psi}_t = [\omega]_L - t \partial \alpha \in H^2(M, L; \R),
\ \ 0 \leq t \leq T.
\end{equation}
Therefore, as long as $\partial \alpha\neq 0$, the number $\df_L (\alpha)$
measures how far one can deform $[\omega]_L$ in a Lagrangian isotopy $\psi$
satisfying \eqref{eqn-omega-psi-t-linear}.

This viewpoint enables us to study $\df_L$ using methods
of ``hard'' symplectic topology. Namely, the existence of
pseudo-holomorphic curves with boundaries on $\psi_t (L)$ for
$0 \leq t \leq T$ may yield constraints on the time-length $T$ of
the deformation, since $[\omega]^{\psi}_t$ evaluates positively on such curves.
All the upper bounds on $\df_L (\alpha)$ known to us and appearing further in this paper
are obtained in this way.

On the other hand, as we explain below, lower bounds on $\df_L$ can be obtained by
``soft" constructions. These bounds come from the study
of the second invariant of $L$, called the {\it Poisson-bracket invariant of}
$L$. It is defined only for $L$ admitting a fibration
over $\SP^1$: it is a function on the set of isotopy classes of smooth
fibrations of $L$ over $\SP^1$.
In the case when $L$ is a Lagrangian torus with a choice of an isotopy class of
its smooth parametrizations $\T^n \to L$, we reduce this invariant to a function
$\bp_L: H^1(L) \to (0,+\infty]$. Postponing the precise definition of
$\bp_L: H^1(L) \to (0,+\infty]$ until Section~\ref{sec-pb-L-intro} we give a
short and informal definition here.

Namely, assume $L$ is a Lagrangian torus equipped with a parametrization
$\T^n \to L$ and $a\in H^1(L)$, $a\neq 0$ (for $a=0$ we set $\bp_L (0):=+\infty$).
Consider a fibration $f: L\to \SP^1$ such that $a$ is the pull-back under $f$ of
the standard generator of $H^1 (\SP^1)$ (the parametrization of $L$ is used
in the construction of $f$ -- see Section~\ref{subsec-defn-of-bpL-for-Lagr-tori}).
Cut $\SP^1$ into four consecutive arcs, denote their preimages under $f$ by
$X_0$, $Y_1$, $X_1$, $Y_0$ (so that $X_0\cap X_1 = Y_0\cap Y_1 =\emptyset$,
$X_0\cup Y_1\cup X_1\cup X_0 = L$) and set $\bp_L (a) := 1/pb_4^+ (X_0,X_1,Y_0,Y_1)$.
Here $pb_4^+$ is the Poisson-bracket invariant of a quadruple of sets defined
in \cite{EP-tetragons} -- it is a refined version of the $pb_4$-invariant
introduced in \cite{BEP} and it admits a dynamical interpretation in terms of the existence of
connecting trajectories of sufficiently small time-length between $X_0$ and $X_1$
for certain Hamiltonian flows -- see Section~\ref{sec-pb-L-intro} for details.
The relation between $\bp_L$ and $\df_L$ is given by the following inequality
(which will be proved in a stronger form in Theorem~\ref{thm-gopbL-vs-dfL}).

\begin{thm}
\label{thm-bpL-leq-dfL-introduction}
$\bp_L\leq \df_L$ on $H^1(L)$.
\end{thm}

This relation, albeit in a different language, was already exploited in \cite{EP-tetragons},
where upper bounds on $\df_L$ were obtained by symplectic rigidity methods in a setting
where the Lagrangian isotopy class of a Lagrangian torus $L$ does not contain (weakly)
exact Lagrangian tori. This was then used to prove the existence of connecting trajectories
of Hamiltonian flows. In this paper we get upper bounds on $\df_L$ in new cases by using
several strong symplectic rigidity results, including some recent ones.
Namely, Theorem~\ref{thm-lagrangian-tori-in-R4} (the original idea of whose proof belongs to E. Opshtein) and
Theorem~\ref{thm-lagr-tori-in-Cn} rely on Gromov's famous work \cite{Gromov},
while Theorem~\ref{thm-tori-in-CPn} and Theorem~\ref{thm-upp-bound-bpL-S2timesS2}
rely on the recent powerful rigidity results of, respectively, K. Cieliebak - K. Mohnke \cite{Ciel-Mohnke}
and G. Dimitroglou Rizell - E. Goodman - A. Ivrii \cite{DRizell-Goodman-Ivrii}.
At the same
time we use new soft dynamical constructions to get lower bounds
on $\bp_L$, and hence on $\df_L$, in many new settings.

In Section~\ref{sec-results-computations} we partially compute the functions
$\bp_L$ and $\df_L$ for several classes of Lagrangian tori. Section~\ref{subsec-discussion}
is then devoted to a discussion of the results and further directions.
In Section~\ref{sec-pb-L-intro} we discuss in detail the definition of $\bp_L$.
The sections following Section~\ref{sec-pb-L-intro} contain the proofs of the results
from Section~\ref{sec-results-computations}.


\section{The main results}
\label{sec-results-computations}

We now present results about $\df_L$ and $\bp_L$ for several examples
of Lagrangian tori in symplectic manifolds.
(For general properties of $\df_L$ and $\bp_L$ see Section~\ref{subsec-pb4-generalities}).

\subsection{Symplectic manifolds without weakly exact Lag\-ran\-gian submanifolds}

Recall that a Lagrangian submanifold $L\subset (M,\omega)$ is called
{\it weakly exact}, if $[\omega]_L \equiv 0$.

\begin{thm}
\label{thm-partial-alpha-proportional-to-omega-upp-bound-on-bpL-alpha}
Let
$L$ be a closed Lagrangian submanifold of a symplectic manifold $(M,\omega)$ (possibly with boundary). Assume that
$(M,\omega)$ does not admit weakly exact Lagrangian submanifolds in the
Lagrangian isotopy class of $L$. Assume that $\alpha \in H^1(L;\R)$ and
\[
\partial \alpha= \frac{1}{C} [\omega]_L
\]
for some $C>0$.

Then $\df_L
(\alpha)\leq C$.
\end{thm}

\begin{proof}
Let $\psi= \{ \psi_t: L\to M\}$, $0\leq t\leq T$, be a Lagrangian isotopy
of $L$ such that
\begin{equation}
\label{eqn-omega-t-omega-L-t-partial-alpha}
[\omega]^\psi_t = [\omega]_L - t\partial \alpha.
\end{equation}
Since, by the hypothesis of the theorem, $\partial \alpha= \frac{1}{C}
[\omega]_L$, we get
\[
[\omega]^\psi_t = [\omega]_L (1-t/C).
\]
Then $t=C$ cannot lie in the interval $[0,T]$, because if it did, we
would have $[\omega]^\psi_C=0$, implying that the Lagrangian submanifold $\psi_C (L)$,
which is Lagrangian isotopic to $L$, is weakly exact,
in contradiction with the hypothesis of the theorem. Hence $T < C$. Since
this is true for any Lagrangian isotopy of $L$ satisfying
\eqref{eqn-omega-t-omega-L-t-partial-alpha}, we get $\df_L (\alpha)
\leq C$.
\end{proof}

A similar result for a particular class of Lagrangian submanifolds and a
particular $\alpha$ was proved in the same way in \cite{EP-tetragons}.

Let us note that symplectic manifolds that do not admit weakly
exact Lagrangian submanifolds at all or in a particular Lagrangian
isotopy class are plentiful and include, in particular, symplectic
vector spaces \cite{Gromov} and complex projective spaces.

\subsection{Lagrangian tori in symplectic surfaces}
\label{subsec-lagr-tori-in-sympl-surfaces-intro}

Suppose $(M^2,\omega)$ is a connected symplectic surface (possibly with
boundary) and $L\subset (M,\omega)$ is a simple closed oriented curve (that is,
a 1-dimensional Lagrangian torus) lying in the interior of $M$. The orientation of $L$
defines an isomorphism $H^1(L) \cong \Z$. This allows to define
$\bp_L$
on $H^1(L)$
(see Section~\ref{subsec-defn-of-bpL-for-Lagr-tori}).
Denote the positive generator of $H^1 (L)$ by $e$.

We distinguish between two possibilities:
when $M \setminus L$ is disconnected and when $M \setminus L$ is
connected. We present precise statements in both cases.

In the first case $L$ divides $M$ into two connected components: $M_+$
and $M_-$ of areas $A_+$, $A_-$, where $0 < A_\pm \leq +\infty$.
The signs $+$ and $-$ here are determined by the usual orientation
convention.

\begin{thm}
\label{thm-pbL-surfaces} For $k \in \N$ we have
\[
\bp_L (k e) = \df_L (k e) = \frac{A_+}{k},
\]
\[
\bp_L (-k e) = \df_L (-k e) = \frac{A_-}{k},
\]
\end{thm}

In the second case let $L \subset (M,\omega)$ be a simple closed
curve such that $M \backslash L$ is connected.

\begin{thm}
\label{thm-surface-separ-curve}
For $k \in \Z$ we have
\[
\bp_L (k e) = \df_L (ke) = +\infty.
\]
\end{thm}

The case $k=1$ was proved in \cite{Samvelyan}. For the proofs of
Theorems~\ref{thm-pbL-surfaces} and \ref{thm-surface-separ-curve}
see Section~\ref{sec-pf-thm-pbL-surfaces}.

\subsection{Toric orbits in symplectic toric manifolds}
\label{subsec-toric-fibers-in-sympl-toric-mfds-intro}

Let $\T^n :=\R^n/\Z^n$. Denote by $\Lie \T^n = \R^n$ the Lie algebra
of $\T^n=\R^n/\Z^n$ and by $\Lied \T^n = (\R^n)^*$ its dual space.
Denote by $(\Z^n)^*$ the integral lattice in $(\R^n)^*$.

Let $(M^{2n},\omega)$ be a connected (not necessarily closed) symplectic
manifold equipped with an effective Hamiltonian action of $\T^n$. Denote
the moment map of the action by $\Phi: M \to (\R^n)^*$. Assume that $\Phi$
is proper, the fibers of $\Phi$ are exactly the orbits of the action and
the image of $\Phi$ is a convex set $\Delta\subset (\R^n)^*$
with non-empty interior so that its interior points are
exactly the regular values of $\Phi$ (by the Atiyah-Guillemin-Sternberg
theorem \cite{Atiyah,Guill-Sternb}, these conditions are automatically
satisfied if $M$ is closed).

Given $\bx\in \Delta$, denote by $L_\bx := \Phi^{-1} (\bx)$ the corresponding
fiber of $\Phi$. If $\bx\in \Int\Delta$, then $L_\bx$ is a Lagrangian torus
and the Hamiltonian $\T^n$-action on $L_\bx$ gives us a preferred isotopy
class of diffeomorphisms $\T^n \to L_\bx$. Thus the pair $H^1 (L_\bx) \subset
H^1 (L_\bx;\R)$ is naturally identified with the pair $(\Z^n)^*\subset(\R^n)^*$.
We denote by
\[
\bp_\bx := \bp_{L_\bx}: (\Z^n)^* \longrightarrow \R
\]
and by
\[
\df_\bx := \df_{L_\bx}: (\R^n)^* \longrightarrow \R
\]
the Poisson bracket and the deformation invariants of $L_\bx$.

For $\bx\in \Int \Delta$ and $\alpha\in H^1 (L;\R)$ define $l_\bx (\alpha)$
as the largest $t > 0$ for which $\bx - t\alpha\in \Delta$ and let $\cI
(\bx,\alpha)$ be the open segment of the open ray $\bx-t\alpha$,
$t\in (0,+\infty)$, connecting $\bx$ and $\bx - l_\bx (\alpha) \alpha$:
\[
\cI (\bx,\alpha) := \{ \bx - t \alpha,\ 0<t< l_\bx (\alpha)\}.
\]
If no such $t$ exists, set $l_\bx (\alpha):= +\infty$ and let $\cI(\bx,\alpha)$
be the whole open ray. In other words, $\cI(\bx,\alpha)$ is the interior of the
closed segment obtained by the intersection of the ray with $\Delta$ and $l_\bx
(\alpha)$ is the ratio of the rational length of $\alpha$ and the rational length
of this segment\footnote{Recall that the rational length of a vector $c v$,
$v\in \Z^n$, is defined as $|c|$.}.

\begin{thm}
\label{thm-toric-case-low-bound}
Let $\alpha \in (\Z^n)^*$. Then
\[
l_\bx (\alpha) \leq \bp_\bx (\alpha).
\]
\end{thm}

For the proof see Section~\ref{sec-pf-thm-toric-case-low-bound}.

\begin{rem}
\label{rem-interior-of-symplectic-toric-mfd}
{\rm
Note that there exists a symplectomorphism $\Phi^{-1} (\Int \Delta)\to \T^n\times \Int\Delta$ that
identifies $\Phi^{-1} (\bx)$ with $\T^n\times  \{ \bx \}$ for each  $\bx\in \Int\Delta$.
Then Theorem~\ref{thm-def-shapes}, applied to the symplectic toric manifold $\Phi^{-1} (\Int \Delta)$,
implies that {\it for $L_\bx$ viewed as a Lagrangian submanifold
of $\Phi^{-1} (\Int \Delta)$}
for each $\alpha$
the deformation invariant
$\df_\bx (\alpha)$  equals $l_\bx (\alpha)$
and thus, by Theorem~\ref{thm-toric-case-low-bound},
\[
\bp_\bx (\alpha) = \df_\bx (\alpha) = l_\bx (\alpha)
\]
for all $\alpha \in H^1(L)$.

For $L_\bx$ viewed as a Lagrangian submanifold of the whole $M$ the problem of
finding $\bp_\bx (\alpha)$ and $\df_\bx (\alpha)$ is more difficult and the results below
that we have been able to obtain are weaker.
}
\end{rem}

\subsection{Lagrangian tori in symplectic vector spaces}
\label{subsec-tori-in-symplectic-vector-sp-intro}

Let $M=\C^n$ be equipped with the standard symplectic structure $\omega$ and
let $z_1,\ldots,z_n$ be the complex coordinates on $\C^n$.
Given $x_1,\ldots,x_n>0$, set $\bx:=(x_1,\ldots,x_n)$. Define a
{\it split Lagrangian torus} $T^n (\bx)\subset \C^n$ by
\[
T^n (\bx) := \{ \pi |z_i|^2 = x_i,\ i=1,\ldots,n\}.
\]
The standard Hamiltonian $\T^n$-action gives us a preferred isotopy
class of diffeomorphisms $\T^n \to T^n (\bx)$ and we naturally identify
$H^1 (L_\bx)\subset H^1 (L_\bx;\R)$ with $(\Z^n)^*\subset(\R^n)^*$.

We first consider the case of Lagrangian tori in $\C^2$.

\bigskip


\noindent {\sc Lagrangian tori in $\C^2$}
\medskip
\medskip

We first present computations of $\df_L$ for general Lagrangian tori in $\C^2$.
We then restrict to the cases of split and Chekanov tori.

\medskip

Let $L \subset (\C^2, \omega)$ be a Lagrangian torus.

We say that an almost complex structure $J$ (on $\C^2$) compatible with
$\omega$ is {\it regular (for $L$) with respect to a point $p\in L$},
if for any $C\in H_2 (\C^2, L)$ the moduli
space of (non-parameterized) somewhere injective $J$-holomorphic disks
in $\C^2$ with boundary in $L$ and with one marked boundary point that
represent the class $C$ {and pass through $p$ (that is, the marked point
coincides with $p$)} is a (transversally cut out) smooth manifold of the expected
dimension. For any $p\in L$ a generic almost complex structure $J$ compatible
with $\omega$ has this property (see Section~\ref{subsec-lagr-tori-in-R4-pf}).

The original idea of the proof of the following theorem belongs to E.Op\-shtein.

\begin{thm}\label{thm-lagrangian-tori-in-R4}
Assume that $H_2(\C^2, L) \simeq \Z \langle A, B \rangle$, where $\omega (A) >0$.
Let $\alpha \in H^1(L; \R)$ so that $\partial \alpha(A) =: \sigma > 0$ and $\partial
\alpha(B) =: \rho$. Assume that for some $k \geq 0$
\[
\mu(A) = 2,\ \mu(B) = 2k,
\]
and
\[
\rho/\sigma\leq k+1\leq \omega (B)/\omega (A).
\]
Then

\medskip
\noindent
(A)
For any $p\in L$ and any almost complex structure $J$ compatible
with $\omega$ and regular with respect to $p$ the mod-2 number
$n_A (p,J)$ of (non-parameterized) somewhere injective $J$-holomorphic disks with boundary
in $L$ in the class $A$ passing through $p$
is well-defined and independent of the choice of $p$ and $J$.

\smallskip
\noindent
(B) If $n_A (p,J)\neq 0$ for some $p$ and $J$ as in (A), then
\[
\df_L(\alpha) \leq \frac{\omega(A)}{\sigma}.
\]
\end{thm}

For the proof see Section~\ref{subsec-lagr-tori-in-R4-pf}.

\begin{rem}
\label{rem-extension-of-thm-lagrangian-tori-in-R4}
{\rm
As it can be seen from the proof, Theorem~\ref{thm-lagrangian-tori-in-R4} remains true
if $\C^2$ is replaced by any 4-dimensional symplectic manifold $(M,\omega)$ which
satisfies $\omega|_{\pi_2 (M)} = c_1|_{\pi_2 (M)} = 0$, is geometrically bounded in the
sense of \cite{ALP-book}, or convex at infinity in the sense of \cite{Eliash-Gromov}, and
$$H_2 (M,L)\cong \Z \langle A,B\rangle \oplus \im (\pi_2 (M) \to \pi_2 (M,L)),$$
where $A,B$ satisfy the same conditions as in Theorem~\ref{thm-lagrangian-tori-in-R4}.
}
\end{rem}

Note that Theorem~\ref{thm-lagrangian-tori-in-R4} applies to certain
split Lagrangian tori $T^2 (\bx)\subset \C^2$. Indeed, by \cite{Cho-Oh},
the standard complex structure $J$ on $\C^2$ (which is, of course, compatible with
$\omega$) is regular (for $T^2 (\bx)$)
with respect to any point $p\in T^2 (\bx)$.
It is also easy to see that for any point $p\in T^2 (\bx)$ there is exactly one regular
(non-parametrized) $J$-holomorphic disk in the class $A$
(with one marked point)
that passes through $p$, if $A$ is any one of the two standard generators
of $H_2(\C^2, T^2 (\bx))$ with positive symplectic area.

Theorem~\ref{thm-lagrangian-tori-in-R4}, together with Theorem~\ref{thm-bpL-leq-dfL-introduction}
and Theorem~\ref{thm-computations-for-split-tori-in-Cn} (this general statement
for split Lagrangian tori in $\C^n$ will appear later),
yields the following corollary for computations of $\df_L$ and $\bp_L$ for split tori in $\C^2$.
We state the result in the case $x_1 \leq x_2$ -- the corresponding result
in the case $x_1 > x_2$ can be deduced from it using the obvious symmetry of $\bp_\bx$ and $\df_\bx$
with respect to permutations of coordinates in $\bx$ (see \eqref{eqn-permutation-of-factors-in-split-torus} below).

\begin{cor}
\label{cor-def-bp-for-split-tori}
Assume $m,n\in\Z$, $\bx=(x_1,x_2)$, $0 < x_1 \leq x_2$. Under these
assumptions the following claims are true:
\begin{enumerate}
\item[(A)] If $m,n \leq 0$, then
\[
\bp_\bx (m,n) = \df_\bx (m,n) = +\infty.
\]
\item[(B)]
If  $x_1 < x_2/ l$ and $1 \leq n \leq l$, then
\[
\bp_\bx (0,n) = \df_\bx (0,n) = +\infty.
\]
\item[(C)] If $n x_1 - m x_2 \leq 0, m>0$, then
\[
x_1/m \leq \bp_\bx (m,n) \leq \df_\bx (m,n).
\]
\item[(D)] If $n x_1 - m x_2\geq 0, n>0$, then
\[
x_2/n \leq \bp_\bx (m,n) \leq \df_\bx (m,n).
\]
\item[(E)] Assume for $\bx = (x_1, x_2)$ that
$2x_1 \leq x_2$. If $m > 0$, $n - 2m \leq 0$, then
\[
x_1/m = \bp_\bx (m,n) = \df_\bx (m,n).
\]
\end{enumerate}
\Qed
\end{cor}

\begin{figure}[h]
    \centering
    \includegraphics[scale=0.72]{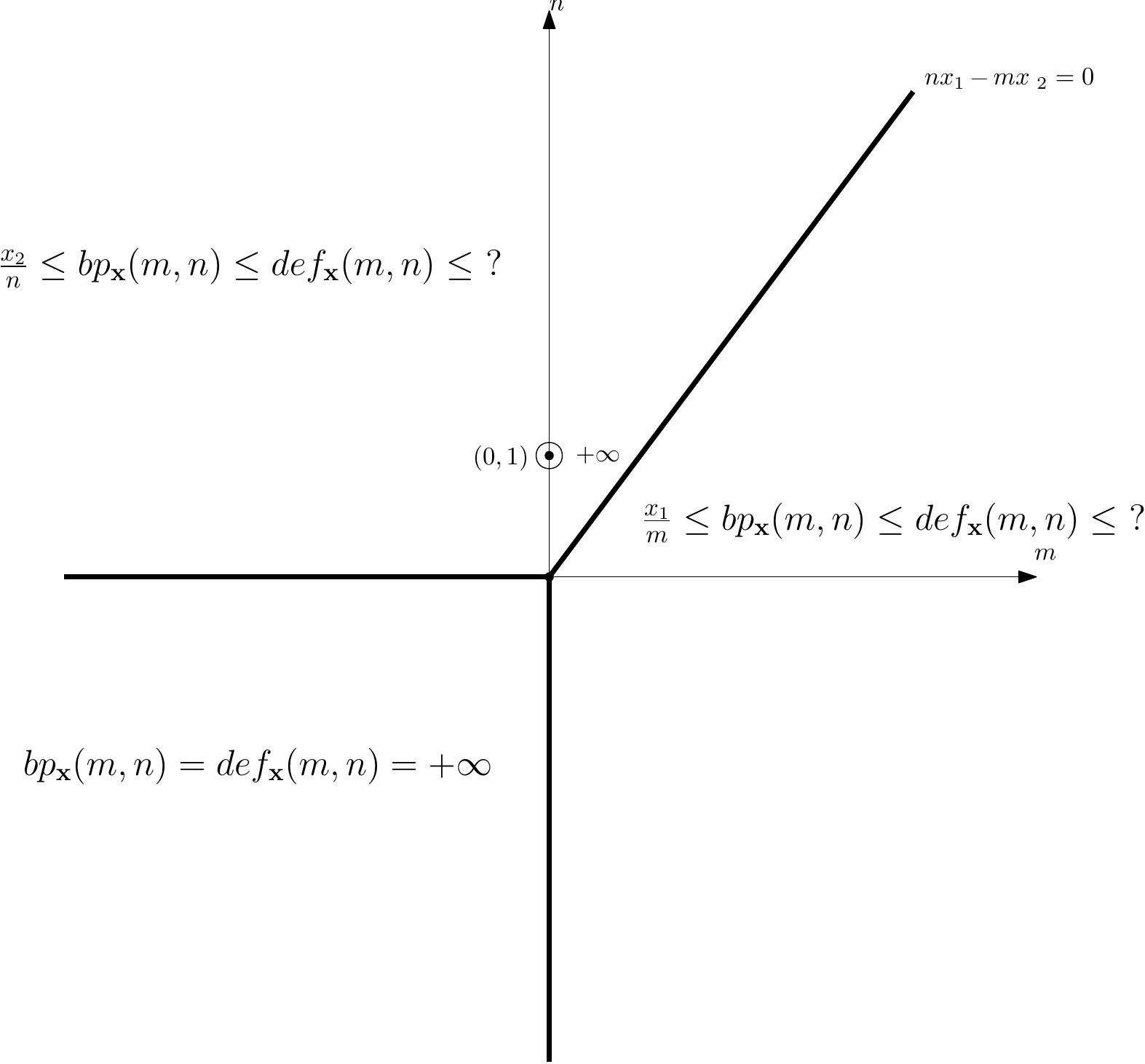}
    \caption{Regions of validity in Corollary~\ref{cor-def-bp-for-split-tori} for $x_1 < x_2$.}
    \label{figure-bp-diagram-2}
\end{figure}

Parts (A) - (D) of Corollary~\ref{cor-def-bp-for-split-tori}
follow directly from Theorems~\ref{thm-bpL-leq-dfL-introduction} and
\ref{thm-computations-for-split-tori-in-Cn}. Part (E) follows from the
inequalities:
\[
x_1/m \leq \bp_\bx (m,n)\leq \df_\bx (m,n)\leq x_1/m.
\]
Here the first inequality follows from part (C) of Corollary~\ref{cor-def-bp-for-split-tori},
the second one from Theorem~\ref{thm-bpL-leq-dfL-introduction}, and the third one from
Theorem~\ref{thm-lagrangian-tori-in-R4} with $A,B$ being the standard basis of
$H_2 (\C^2, T^2 (\bx))$ (so that $\omega(A) = x_1$, $\omega(B) = x_2$),
$\alpha = me_1+ne_2$, $a=m$, $b=n$ and $k = 1$.

\begin{figure}[h]
    \centering
    \includegraphics[scale=0.72]{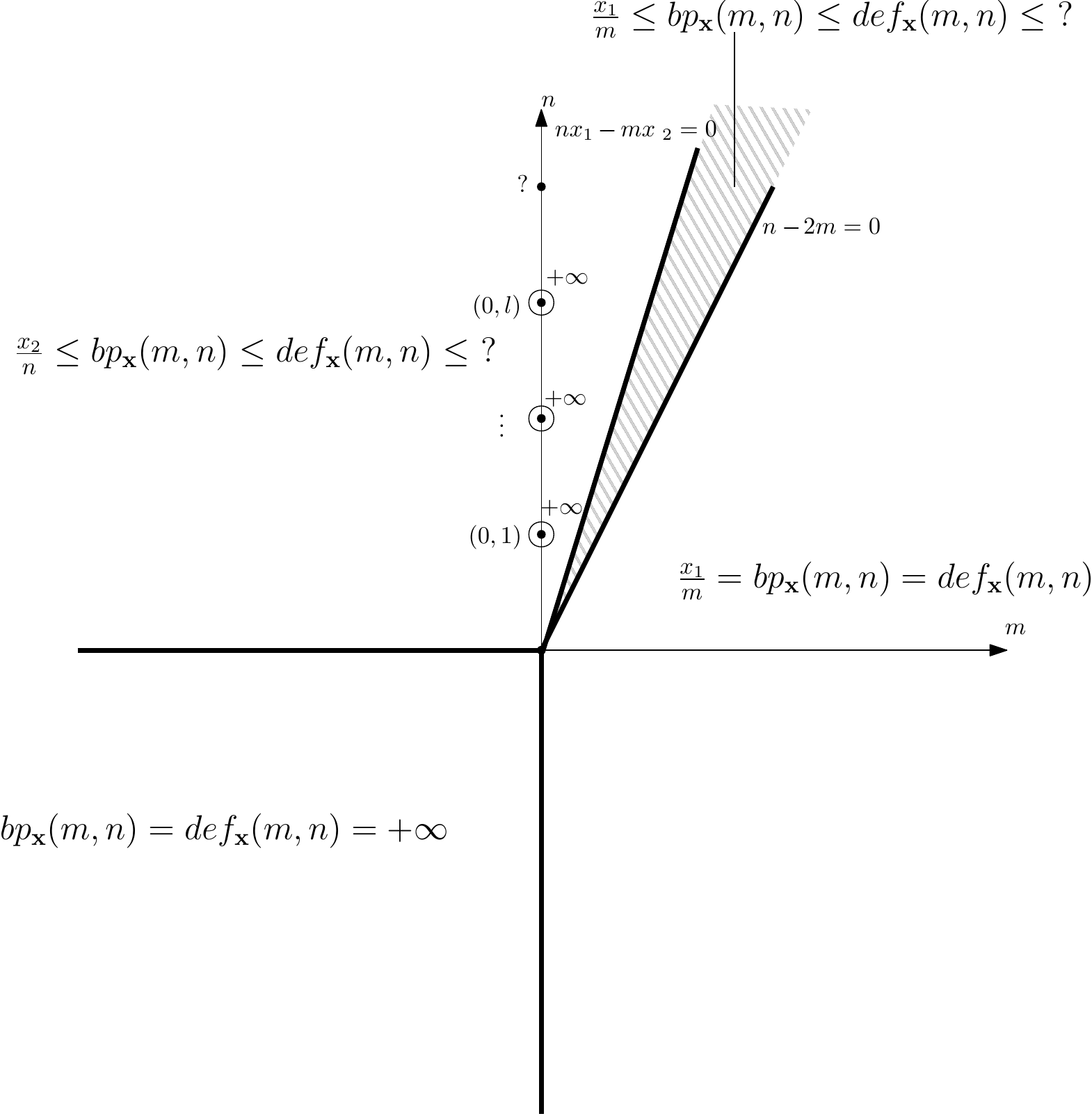}
    \caption{Regions of validity in Corollary~\ref{cor-def-bp-for-split-tori} for $x_1 < x_2/l$ for $l \geq 2$.}
    \label{figure-bp-diagram-22}
\end{figure}

In the case of the split monotone Lagrangian torus in $\C^2$
Corollary~\ref{cor-def-bp-for-split-tori}, together with the obvious
homogeneity property of $\df_\bx$, $\bp_\bx$ with respect to $\bx$ (see part (B) of Proposition~\ref{prop-basic-properties-of-pbL-for-split-tori-in-R2n}
below) and Theorem~\ref{thm-partial-alpha-proportional-to-omega-upp-bound-on-bpL-alpha}
(since $\C^2$ does not admit weakly exact Lagrangian submanifolds by \cite{Gromov}), yields
the following result.

\begin{cor}
\label{cor-def-bp-for-split-monotone-tori-in-C2}
Assume $m,n\in\Z$, $\bx=(x,x)$, $x>0$. Then
\begin{alignat*}{3}
+\infty &= \bp_\bx (m,n) && = \df_\bx (m,n),\ \textit{if}\ m,n \leq 0, \\
\frac{x}{m}  &= \bp_\bx (m,m) &&= \df_\bx (m,m),\ \textit{if}\ m=n > 0, \\
\frac{x}{\max\{ m,n\}} &\leq \bp_\bx (m,n) &&\leq \df_\bx (m,n),\ \textit{if}\ m > 0\ \textit{or}\ n > 0.
\end{alignat*}
\Qed
\end{cor}

\begin{figure}[h]
    \centering
    \includegraphics[scale=0.72]{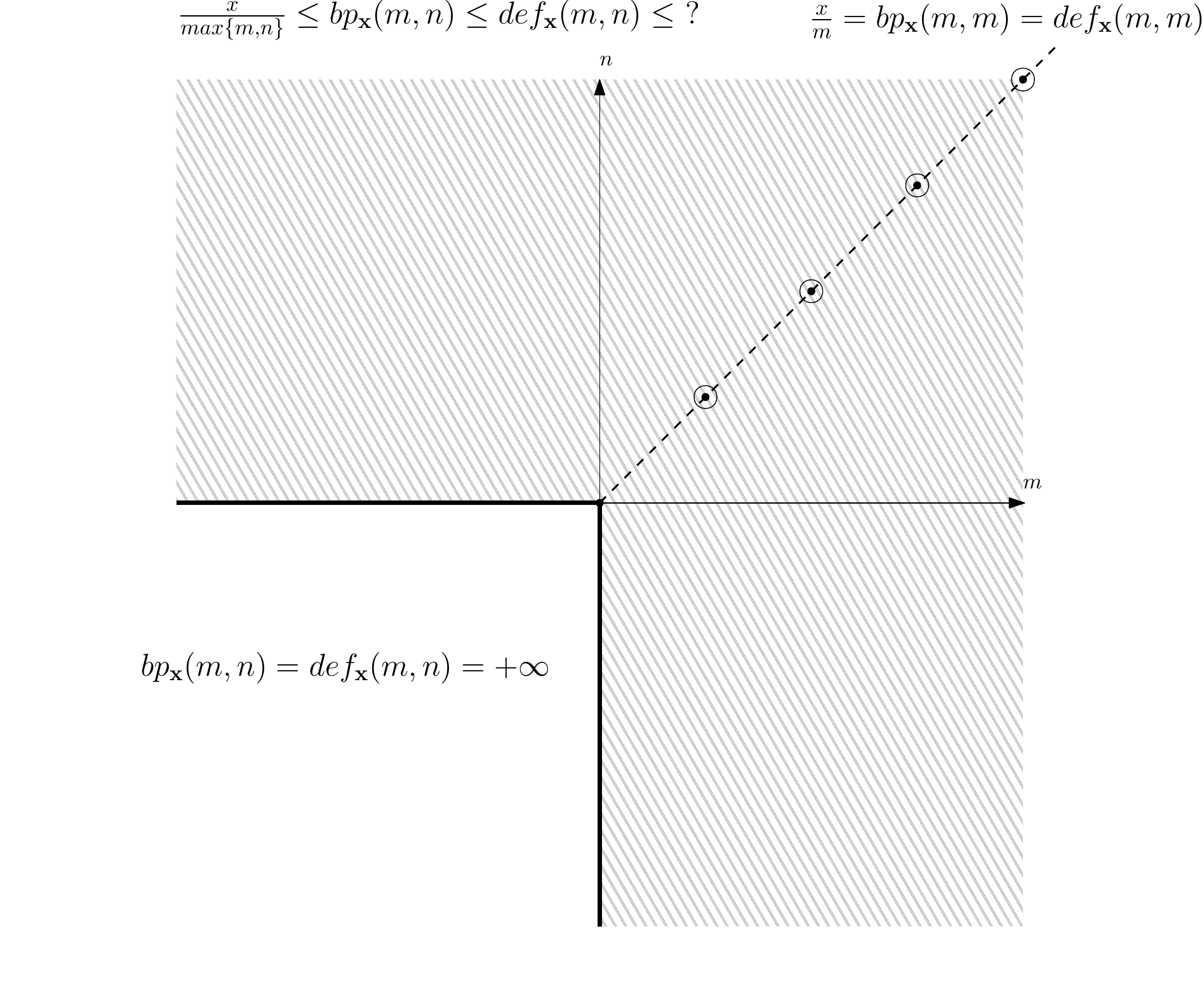}
    \caption{Regions of validity in Corollary~\ref{cor-def-bp-for-split-monotone-tori-in-C2} for monotone Lagrangian tori.}
    \label{figure-bp-diagram-3}
\end{figure}

\vfil
\eject
\bigskip


\noindent{\sc Chekanov tori in $\C^2$}
\medskip
\medskip

Chekanov tori $\Theta_a$, $a>0$, in $\C^2$ were originally introduced
in \cite{Chekanov-tori-in-R2n} (cf. \cite{El-Polt-survey}). The
torus $\Theta_a\subset \C^2=\C\times\C$ is defined as follows:
consider the first open quadrant $Q$ of $\C$ and a point $q\in Q$.
Fix a foliation of $Q\setminus q$
by simple closed curves (each winding once around $q$) so that for each $a>0$
there is exactly one curve in the foliation that bounds a disk of area $a$ in $Q$.

Pick $a>0$ and let $\eta (t)$ be the counterclockwise regular parameterization of the corresponding curve in the foliation.
Then
\[
\Theta_a := \Bigg\{ \frac{1}{\sqrt{2}} \bigg( e^{2\pi i s} \eta
(t), e^{-2\pi i s} \eta (t)\bigg)\Bigg\}.
\]
Consider the basis $\Gamma,\gamma$ of $H_1 (\Theta_a)$, where
$\Gamma,\gamma$ are, respectively, the homology classes of the
curves $t\mapsto \frac{1}{\sqrt{2}} (\eta (t), \eta (t))$ and
$s\mapsto\tfrac{1}{\sqrt{2}} (e^{2\pi i s} \eta (0), e^{-2\pi i s} \eta
(0))$ in $\Theta_a$. Let $\hGamma,\hgamma$ be the integral basis of
$H^1 (\Theta_a)$ dual to $\Gamma,\gamma$. Define the functions
$\df_a: (\Z^2)^* \to (0,+\infty]$,
$\bp_a: (\Z^2)^* \to (0,+\infty]$ by
\[
\df_a (m,n) := \df_{\Theta_a} (m\hGamma+n\hgamma),
\ \ \bp_a (m,n) := \bp_{\Theta_a} (m\hGamma+n\hgamma).
\]

\begin{thm}
\label{thm-chekanov-tori-v2}
The following claims hold for the Chekanov
tori in $\C^2$:

\smallskip
\noindent
(A) $\bp_a (m, n) \geq a/m$, if $m > 0$.

\smallskip
\noindent
(B) $\bp_a (m, n) = \df_a (m,n)=  +\infty$ for $m < 0$.
\end{thm}

For the proof see Section~\ref{subsec-pf-thm-chekanov-tori-v2}.

\bigskip
\noindent{\sc Lagrangian tori in $\C^n$ for general $n\in\N$}
\medskip
\medskip

We first generalize the statement of Theorem~\ref{thm-lagrangian-tori-in-R4}
to Lagrangian tori in $\C^n$ for a general $n \in \N$.

Let $L \subset (\C^n, \omega)$ be a Lagrangian torus.

As above, we say that an almost complex structure $J$ (on $\C^n$) compatible
with $\omega$ is {\it regular (for $L$) with respect to a point $p\in L$}, if for any $C\in H_2 (\C^n, L)$ the moduli
space of non-parameterized somewhere injective $J$-holomorphic disks in $\C^n$ with boundary on
$L$ and with one marked point that represent the class $C$ and pass through $p$ (that is, the marked
point coincides with $p$) is a (transversally cut out)
 smooth manifold of the expected dimension. A generic $J$ (on $\C^n$) compatible
with $\omega$ has this property (see Section~\ref{sec-lagr-tori-in-Cn-proofs}).

\begin{thm}
\label{thm-lagr-tori-in-Cn} \
Assume
$H_2(\C^n, L) \simeq \Z \langle A_1, \ldots, A_n \rangle$, where
\[
\mu(A_i) = 2, \quad i = 1, \ldots, n,
\]
and
\[
\omega(A_1) =: a > 0, \quad \omega(A_2) = \ldots = \omega(A_n) =: b.
\]
Let $\alpha \in H^1(L; \R)$ so that
\[
\partial\alpha(A_1) =: \sigma > 0,\quad \partial\alpha(A_i) =: \rho,\;\; i = 2, \ldots, n.
\]
Assume
\[
\rho/\sigma \leq \frac{n+2}{2} \leq b/a, \ \textit{if}\ n\ \textit{is even},
\]
\[
\rho/\sigma \leq \frac{n+3}{2} \leq b/a, \ \textit{if}\ n\ \textit{is odd}.
\]
Then

\medskip
\noindent
(A) For an almost complex structure $J$ compatible with $\omega$ and regular
(for $L$) with respect to a point $p\in L$ the mod-2 number $n_{A_1} (p,J)$ of
(non-parameterized somewhere injective) $J$-holomorphic disks with boundary
in $L$ in the class $A_1$ with one marked point passing through $p$ is
well-defined and independent of the choice of $p$ and $J$.

\smallskip
\noindent
(B) If $n_{A_1} (p,J)\neq 0$ for some $p$ and $J$ as in (A), then
\[
\df_L(\alpha) \leq \frac{\omega(A_1)}{\sigma} = \frac{a}{\sigma}.
\]
\end{thm}

For the proof see Section~\ref{sec-lagr-tori-in-Cn-proofs}.

\begin{rem}
\label{rem-thm-lagrangian-tori-in-R4-R2n}
{\rm
For Lagrangian tori $L\subset \C^2$ satisfying
$H_2(\C^2, L) \simeq \Z \langle A_1, A_2 \rangle$, where
$\mu(A_1) = \mu (A_2) = 2$, Theorem~\ref{thm-lagr-tori-in-Cn}
gives the same result as Theorem~\ref{thm-lagrangian-tori-in-R4}.
}
\end{rem}

\begin{rem}
\label{rem-extension-of-thm-lagrangian-tori-in-R2n}
{\rm
As it can be seen from the proof, and similarly to Remark~\ref{rem-extension-of-thm-lagrangian-tori-in-R4},
Theorem~\ref{thm-lagr-tori-in-Cn} remains true if $\C^n$ is replaced by any $2n$-dimensional symplectic
manifold $(M,\omega)$ which satisfies $\omega|_{\pi_2 (M)} = c_1|_{\pi_2 (M)} = 0$, is geometrically
bounded in the sense of \cite{ALP-book}, or convex at infinity in the sense of \cite{Eliash-Gromov}, and
the following holds:
\begin{equation}
\label{eqn-extension-of-thm-lagrangian-tori-in-R2n}
H_2(M, L) \simeq \Z \langle A_1, \ldots, A_m \rangle \oplus \im (\pi_2 (M) \to \pi_2 (M,L)),
\end{equation}
for some $m\in\N$ (not necessarily equal to $n$!), where
\[
\mu(A_i) = 2, \quad i = 1, \ldots, m,
\]
\[
\omega(A_1) =: a > 0, \quad \omega(A_2) = \ldots = \omega(A_m) =: b,
\]
and $\alpha \in H^1(L; \R)$ satisfies
\[
\partial\alpha(A_1) =: \sigma > 0,\quad \partial\alpha(A_i) =: \rho, \;\;i = 2, \ldots, m,
\]
\[
\rho/\sigma \leq \frac{n+2}{2} \leq b/a, \ \textit{if}\ n\ \textit{is even},
\]
\[
\rho/\sigma \leq \frac{n+3}{2} \leq b/a, \ \textit{if}\ n\ \textit{is odd}.
\]

In such a case Theorem~\ref{thm-lagr-tori-in-Cn} can be applied not only to $M,L$ and
$\alpha$ but also to $(\hM:= M\times T^* \SP^1, \homega:= \omega\oplus d\theta \wedge dr)$
(where $\theta\in \SP^1, r\in \R$ are the standard coordinates on $T^* \SP^1$),
$\hL:= L\times \SP^1\subset \hM$ (where $\SP^1 = \{r=0\}$ is the zero-section of $T^* \SP^1$)
and $\halpha$, which is the image of $\alpha$ under the inclusion $H^1 (L;\R)\to H^1 (\hL;\R)
= H^1 (L;\R) \oplus\R$. (Note that $H_2(M, L) \simeq H_2 (\hM,\hL)$).

Moreover, if an almost complex structure $J$ on $M$ is compatible with $\omega$ and regular
(for $L$) with respect to a point $p\in L$, then $\hJ := J \oplus j$ (where $j$ is the standard
complex structure on $T^* \SP^1$) is an almost complex structure on $\hM$ compatible with
$\homega$ and regular (for $\hL$) with respect to a point $\hp := p\times q \in \hL = L\times \SP^1$
(where $q$ is a point in $\SP^1$), and $n_{A_1} (p,J) = n_{A_1} (\hp, \hJ)$.

Then, if $\alpha\in H^1 (L)$, we can apply Theorem~\ref{thm-lagr-tori-in-Cn} to $\hM, \hL, \halpha$
and, together with Theorem~\ref{thm-bpL-leq-dfL-introduction}, it yields
\[
\bp_\hL(\halpha) \leq \df_\hL(\halpha) \leq \frac{a}{\sigma},
\]
or, equivalently,
\[
pb_4^+ (X_0\times \SP^1, Y_0\times \SP^1, X_1\times \SP^1, Y_1\times \SP^1) \geq \frac{\sigma}{a},
\]
where $X_0,X_1, Y_0,Y_1$ are the closed subsets of $L= X_0\cup X_1\cup Y_0\cup Y_1$ used
to define $\bp_L(\alpha)$.

By \cite{EP-tetragons} (cf. Theorem~\ref{thm-pb4-dynamical-meaning} below), this yields the
existence of connecting trajectories from $X_0$ to $X_1$ for Hamiltonian flows (defined for all times)
generated by time-periodic Hamiltonians $H: M\times \SP^1\to \R$ such that $\Delta_H
:= \min_{Y_1\times \SP^1} H - \max_{Y_0\times \SP^1} H  > 0$. Such a connecting trajectory will
have time-length $\displaystyle \leq \frac{a}{\sigma\Delta_H}$.
}
\end{rem}

Note that Theorem~\ref{thm-lagr-tori-in-Cn} applies to certain split Lagrangian tori in
$(\C^n,\omega)$ and the standard basis $A_1, \ldots, A_n$ of
$H_2 (\C^n, T^n(\bx))$ -- indeed, for the standard complex structure $J$ on $\C^n$
(which is, of course, compatible with $\omega$) there is exactly one (non-parametrized)
$J$-holomorphic disk in the class $A_1$ passing through any point of $T^n(\bx)$.
The regularity of $J$ (for $T^n(\bx)$) with respect to any $p\in T^n(\bx)$ follows again from \cite{Cho-Oh}.

The following results give partial information about the function
$\bp_{\bx}$ for split Lagrangian tori. Set $x_{min} := \min\ \{x_1,\ldots,x_n\}$. Denote by
$e_1,\ldots,e_n$ the standard generators of $H^1 (T^n(\bx))\cong
(\Z^n)^*$.

\begin{thm}
\label{thm-computations-for-split-tori-in-Cn} \

\smallskip
\noindent
(A) If $m_i \leq 0$ for all $i=1,\ldots,n$, then
\[
\bp_{\bx}(m_1, \ldots, m_n)=\df_{\bx}(m_1, \ldots, m_n) = +\infty.
\]
Otherwise
\[
\min_{i, m_i>0}\ x_i/m_i \leq \bp_{\bx}(m_1, \ldots, m_n).
\]

\smallskip
\noindent
(B) If $x_1=\ldots=x_n=:x$ and $k \in \N$, then
\[
\bp_{\bx}(k,\ldots,k) = \df_{\bx}(k,\ldots,k)  = x/k.
\]

\smallskip
\noindent
(C) Let $l \in \N$ and assume $x_{min} < x_i / l$. Then for all
$k \in \N$ with $k \leq l$,
\[
\bp_{\bx}(k e_i) = \df_{\bx}(k e_i) = +\infty.
\]

\end{thm}

For the proof see
Section~\ref{sec-pf-thm-computations-for-split-tori-in-Cn}.

Combining Theorem~\ref{thm-lagr-tori-in-Cn} and
Theorem~\ref{thm-computations-for-split-tori-in-Cn} we obtain
for certain split tori $T^n(\bx)$ the following result.

\begin{cor}
Assume $x_2 = \ldots = x_n =: y$ and let $0 < x_1 \leq 2y/(n+2)$
if $n$ is even and $0 < x_1 \leq 2y/(n+1)$ if $n$ is odd.
For $m_1 \in \N$ let the integer $m_2 = \ldots = m_n \leq y\, m_1 /x_1$.
Then we have
\[
\bp_{\bx}(m_1, \ldots, m_n) = \df_{\bx}(m_1, \ldots, m_n) = x_1/m_1.
\]
In particular, for any $m_1\in \N$
\[
\bp_{\bx}(m_1 e_1) = \df_{\bx}(m_1 e_1)= x_1/m_1.
\]
Furthermore, assume for $l \in \N$ that $x_1 < y/l$. Then for all
$1 \leq k \leq l$ and all $2 \leq i \leq n$ we have
\[
\bp_{\bx}(k e_i) = \df_{\bx}(k e_i) = +\infty.
\]
\end{cor}

\subsection{Lagrangian tori in complex projective spaces}
\label{subsec-cpn}

Let $M=\C P^n$ and let $\omega$ be the standard Fubini-Study
symplectic form on $\C P^n$ normalized so that $\int_{\C P^1}
\omega = 1$.

\begin{thm}\label{thm-tori-in-CPn}
Let $L \subset (\C P^n, \omega)$ be a Lagrangian torus and
$\alpha \in H^1(L)$. Consider $[\omega]_t := [\omega]_L -
t \partial\alpha \in H^2(\C P^n, L; \R)$ for $t \geq 0$.
If there exists a $C > 0$ such that $[\omega]_C \in
H^2(\C P^n, L; \frac{1}{n}\Z)$, then
\[
\df_L(\alpha) \leq C.
\]
\end{thm}

For the proof see Section~\ref{sec-pf-thm-cpn-toric-fiber-upp-bound}.

We provide a precise statement in the case of Lagrangian
torus fibers. Consider the standard Hamiltonian $\T^n$-action
on $\C P^n$ and denote its moment map by $\Phi :\C P^n\to (\R^n)^*$.
Its image is the simplex
\[
\Delta := \{ (x_1,\ldots,x_n)\in (\R^n)^*\ |\ x_1, \ldots,
x_n\geq 0,\ 0\leq x_1+\ldots+x_n\leq 1\}.
\]
As in Section~\ref{subsec-toric-fibers-in-sympl-toric-mfds-intro},
for $\bx\in\Int \Delta$ denote $L_\bx :=\Phi^{-1} (\bx)$ the
corresponding Lagrangian torus fiber of $\Phi$ and set
\[
\df_\bx (\alpha) := \df_{L_\bx} (\alpha),\ \ \bp_\bx (\alpha)
:= \bp_{L_\bx} (\alpha)
\]
for each $\alpha \in (\Z^n)^* \cong H^1(L_\bx)$.

By Theorem~\ref{thm-toric-case-low-bound}, for all $\alpha \in
(\Z^n)^*$
\begin{equation}
\label{eqn-toric-fiber-low-bound}
l_\bx (\alpha) \leq \bp_\bx (\alpha).
\end{equation}

For certain $\alpha\in (\Z^n)^*$ and $\bx\in \Int \Delta$ we
obtain an upper bound on $\df_\bx (\alpha)$ from
Theorem~\ref{thm-tori-in-CPn}.

Namely, define $d_\bx (\alpha)$ as the smallest $t > 0$ for
which $\bx - t\alpha\in \frac{1}{n}\cdot (\Z^n)^*$ and let
$\cJ(\bx,\alpha)$ be the open segment of the same open ray
$\bx-t\alpha$, $t\in (0,+\infty)$, as above connecting $\bx$
and $\bx - d_\bx (\alpha) \alpha$:
\[
\cJ (\bx,\alpha) := \{ \bx - t \alpha,\ 0<t< d_\bx (\alpha)\}.
\]
If no such $t$ exists, set $d_\bx (\alpha):= +\infty$ and let $\cJ
(\bx,\alpha)$ be the whole open ray. In other words, $\cJ(\bx,
\alpha)$ is the open part of the segment of the ray connecting
the origin $\bx$ of the ray to the closest point of the lattice
$\frac{1}{n}\cdot (\Z^n)^*$ on the ray.

\begin{cor}
\label{thm-cpn-toric-fiber-upp-bound} With the setup as above,
\begin{equation}
\label{eqn-cpn-toric-fiber-upp-bound}
\df_\bx (\alpha) \leq d_\bx (\alpha).
\end{equation}
\Qed
\end{cor}

In case $l_\bx (\alpha) = d_\bx (\alpha)$ the lower and the upper
bounds \eqref{eqn-toric-fiber-low-bound},
\eqref{eqn-cpn-toric-fiber-upp-bound} yield $\df_\bx (\alpha) =
d_\bx (\alpha) = l_\bx (\alpha) = \bp_\bx (\alpha)$. For instance,
this happens when the intervals $\cI (\bx,\alpha)$ and $\cJ
(\bx,\alpha)$ coincide. Thus, we get the following corollary.

\begin{cor}
\label{cor-cpn-toric-fiber-upper-and-lower-bounds-coincide}

Assume that $\bx=\kappa \alpha$ for some $\kappa>0$, $\alpha \in
H^1 (L_\bx)$, and that the open interval $\{ t\bx,\ 0 < t < \kappa \}$,
does not contain points of the lattice $\frac{1}{n}\cdot (\Z^n)^*$.

Then this interval coincides with $\cI (\bx,\alpha)$ and $\cJ
(\bx,\alpha)$ and therefore
\[
\df_\bx (\alpha) = \bp_\bx (\alpha) = l_\bx (\alpha) = d_\bx(\alpha)
= \kappa.
\]
\Qed
\end{cor}

\subsection{Lagrangian tori in $\SP^2\times \SP^2$}
\label{subsec-Lagr-tori-in-S2timesS2}

Let $(\SP^2,\sigma)$ be the standard symplectic sphere with normalized
symplectic area $\int_{\SP^2} \sigma = 1$.

\begin{thm}\label{thm-upp-bound-bpL-S2timesS2}
Let $L \subset (\SP^2 \times \SP^2, \sigma \oplus \sigma)$ be a Lagrangian
torus and $\alpha \in H^1(L)$. Consider $[\omega]_t := [\omega]_L -
t \partial\alpha \in H^2(\SP^2 \times \SP^2, L; \R)$ for $t \geq 0$.
If there exists a $C > 0$ such that $[\omega]_C \in H^2(\SP^2 \times \SP^2,
L; \Z)$, then
\[
\df_L(\alpha) \leq C.
\]
\end{thm}

For the proof see Section~\ref{sec-pf-thm-S2-times-S2-upp-bound}.

We provide a precise statement in the case of Lagrangian
torus fibers. Consider the standard Hamiltonian $\T^2$-action
on $\SP^2 \times \SP^2$ and denote its moment map by
$\Phi : \SP^2 \times \SP^2
\to (\R^2)^*$. Its image is given by
$\Delta := [0,1] \times [0,1]$.
As in Section~\ref{subsec-toric-fibers-in-sympl-toric-mfds-intro},
for $\bx\in\Int \Delta$ denote $L_\bx :=\Phi^{-1} (\bx)$ the
corresponding Lagrangian torus fiber of $\Phi$ and set
\[
\df_\bx (\alpha) := \df_{L_\bx} (\alpha),\ \ \bp_\bx (\alpha)
:= \bp_{L_\bx} (\alpha)
\]
for each $\alpha \in (\Z^2)^* \cong H^1(L_\bx)$.

By Theorem~\ref{thm-toric-case-low-bound}, for all $\alpha \in
(\Z^2)^*$
\[
l_\bx (\alpha) \leq \bp_\bx (\alpha).
\]

For certain $\alpha\in (\Z^2)^*$ and $\bx\in \Int \Delta$ we
obtain an upper bound on $\df_\bx (\alpha)$ from
Theorem~\ref{thm-upp-bound-bpL-S2timesS2}.

Namely, define $\rho_\bx (\alpha)$ as the smallest $t > 0$ for
which $\bx - t\alpha\in (\Z^2)^*$ and let
$\cK(\bx,\alpha)$ be the open segment of the same open ray
$\bx-t\alpha$, $t\in (0,+\infty)$, as above connecting $\bx$
and $\bx - \rho_\bx (\alpha) \alpha$:
\[
\cK (\bx,\alpha) := \{ \bx - t \alpha,\ 0<t< \rho_\bx (\alpha)\}.
\]
If no such $t$ exists, set $\rho_\bx (\alpha):= +\infty$ and let $\cK
(\bx,\alpha)$ be the whole open ray. In other words, $\cK(\bx,
\alpha)$ is the open part of the segment of the ray connecting
the origin $\bx$ of the ray to the closest point of the lattice
$(\Z^2)^*$ on the ray.

\begin{cor}
\label{cor-S2xS2-toric-fiber-upp-bound} With the setup as above,
\[
l_\bx (\alpha) \leq \bp_\bx (\alpha) \leq \df_\bx (\alpha) \leq
\rho_\bx (\alpha).
\]
In particular, if $\cK (\bx,\alpha)$ connects $\bx$ to either of the
four vertices of $\Delta$, then $\cI (\bx,\alpha) = \cK (\bx,\alpha)$
then
\[
l_\bx (\alpha) = \bp_\bx (\alpha) = \df_\bx (\alpha) = \rho_\bx (\alpha).
\]
\Qed
\end{cor}

\subsection{Discussion and open questions}
\label{subsec-discussion}

The results above reflect first steps in the study of the invariants $\bp_L$
and $\df_L$. In this section we discuss the main difficulty in the current approach
and a possible direction of further investigation of these invariants.

As we have already mentioned in the introduction, the lower bounds on
$\bp_L$ come from ``soft" constructions, while the upper bounds on $\df_L$
are based on ``rigid" symplectic methods -- foremost, on strong results
yielding the existence of pseudo-holomorphic disks with boundary on
Lagrangian submanifolds appearing in appropriate Lagrangian isotopies
of $L$. Unfortunately, it seems that these strong rigidity results are not
strong enough to get upper bounds on $\df_L (\alpha)$ for many $\alpha$
even for the basic examples of Lagrangian tori considered above. The
(well-known) difficulty comes from the fact that
the pseudo-holomorphic disks in a given relative homology class of $L$ in
$M$ may not persist in a Lagrangian isotopy $\psi=\{ \psi_t: L\to M\}$ of $L$
(since bubbling-off of pseudo-holomorphic disks is a codimension-1 phenomenon),
which makes it very difficult to track, as $t$ changes, the relative homology classes
of $\psi_t (L)$ carrying the disks and, accordingly, the symplectic areas of these
disks. (As above, we use the Lagrangian isotopy to identify the relative homology
groups of all $\psi_t (L)$).
In our case the difficulty is compounded by the need to track the pseudo-holomorphic
disks and their areas for an {\it arbitrary} Lagrangian isotopy $\psi$ of $L$ satisfying
the cohomological condition \eqref{eqn-flux-path-linear}, with no
{\it a priori} geometric information about it. Such a Lagrangian isotopy $\psi$ typically
involves non-monotone Lagrangians which limits
even more the control over the disks and their areas.

Thus, the progress on upper bounds for $\df_L$ depends on getting more precise
information on pseudo-holomorphic/symplectic disks with boundary on
(possibly non-monotone) Lagrangian submanifolds Lagrangian isotopic to $L$.

Here is an example of possible additional helpful information on the disks.
We will present it in the case of Lagrangian tori in the standard symplectic $\C^2$.

Assume there is a way to associate to any non-monotone Lagrangian torus $L\subset \C^2$
an ordered integral basis $B_L$ of $H_2 (\C^2,L)$ with the following properties:

\medskip
\noindent
(a) For any Lagrangian isotopy $L_t$ of $L$ {\it among non-monotone Lagrangian tori}
the following conditions hold:

\begin{itemize}

 \item{} the bases $B_{L_t}$ for different $t$ are all identified with each other under
 the isomorphisms between the groups $H_2 (\C^2,L_t)$ defined by the Lagrangian isotopy;

\item{} the symplectic areas of the elements of $B_{L_t}$ are positive and change continuously with $t$.

\end{itemize}

\medskip
\noindent
(b) For $L=T^2(x_1,x_2)$, $x_1<x_2$, the basis $B_L$ is the standard basis of $H_2 (\C^2,T^2(x_1,x_2))$.

\medskip
If such bases exist, a rather straightforward argument would allow to strengthen
Corollary~\ref{cor-def-bp-for-split-tori} and show that if $0 < x_1 \leq x_2$ and
$n x_1 - m x_2 \leq 0$, $m>0$, then
\[
x_1/m = \bp_\bx (m,n) = \df_\bx (m,n).
\]
This, in turn, would allow to strengthen Corollary~\ref{cor-def-bp-for-split-monotone-tori-in-C2}
and show that for a monotone split Lagrangian torus $T^2 (x,x)$ one has
\[
\bp_\bx (m,n) = \df_\bx (m,n) = \frac{x}{\max \{ 0,m,n\}}
\]
for all $m,n\in\Z$.

The question about the existence of a basis $B_L$ is motivated by the folklore
conjecture that any non-monotone torus $L$ in $\C^2$ is Hamiltonian isotopic to
a split torus $T^2(x_1,x_2)$.

Indeed, assume the conjecture is true. Then, by a theorem of Y.Chekanov
\cite{Chekanov-tori-in-R2n}, the ordered pair $(x_1,x_2)$ is uniquely determined by $L$,
as long as we require $x_1<x_2$. A Hamiltonian isotopy between $L$ and $T^2(x_1,x_2)$
identifies the standard basis of $H_2 (\C^2,T(x_1,x_2))$ with an integral basis $B_L$
of $H_2 (\C^2,L)$. The basis $B_L$ does not depend on the choice of the Hamiltonian
isotopy (since any Hamiltonian isotopy of $\C^2$ preserving $T^2 (x_1,x_2)$ as a set
acts trivially on the homology of $T^2 (x_1,x_2)$, by a theorem of M.-L.Yau \cite{YauML}).
It is not hard to check that $B_L$ satisfies (a) and (b).

The existence of bases $B_L$ satisfying (a) and (b) is, of course, much weaker than
the conjecture and, accordingly, might be easier to prove and to generalize to higher
dimensions.

\section{The invariant $\bp_L$ and its properties}
\label{sec-pb-L-intro}

The definition of $\bp_L$ is based on the following construction related to
the Poisson bracket.

\subsection{Poisson bracket invariants}
\label{subsec-poisson-bracket-invariants-definition}

Let $(M^{2n},\omega)$ be a connected symplectic manifold,
possibly with boundary. Let $C^\infty (M)$ and $C^\infty_c (M)$ denote,
respectively, the spaces of all and of compactly supported smooth functions
on $M$ (in the latter case the support is allowed to intersect the boundary of $M$).

Our
sign convention for the Poisson bracket on $(M,\omega)$ will be as
follows. For $G\in C^\infty (M)$ define a vector field $\sgrad G$ by
$i_{\sgrad G} \omega = -dG$. Given $F,G\in C^\infty (M)$, define the
Poisson bracket $\{F,G\}$ by $$\{F,G\} := \omega(\sgrad G,\sgrad F) =
dF (\sgrad G) = - dG (\sgrad F) =$$
$$= L_{\sgrad G} F = - L_{\sgrad F} G.$$

We say that sets $X_0$, $X_1$, $Y_0$, $Y_1\subset M$ form an {\it
admissible quadruple}, if they are compact and $X_0\cap X_1 = Y_0\cap
Y_1 = \emptyset$.

Assume $X_0, X_1, Y_0, Y_1 \subset M$ is an admissible quadruple.
Recall from \cite{EP-tetragons} (cf. \cite{BEP}) the following
definition:
\[
pb_4^{+} (X_0,X_1,Y_0,Y_1) := \inf_{(F,G)\in\cF} \max_M\ \{F,G\},
\]
where $\cF = \cF (X_0,X_1,Y_0,Y_1)$ is the set of all pairs $(F,G)$,
$F,G \in C^\infty_c (M)$, such that
\begin{equation}
\label{eqn-F-G-X-Y-ineqs-defn-pb4} F|_{X_0}\leq 0, \ F|_{X_1} \geq
1,\ G|_{Y_0}\leq 0,\ G|_{Y_1}\geq 1.
\end{equation}
One can show (see \cite{EP-tetragons}, \cite{BEP}) that $\cF$ can be
replaced in the definition of $pb_4^{+}$ by a smaller set $\cF' =
\cF' (X_0,X_1,Y_0,Y_1)$ of pairs $(F,G)$, $F,G\in C^\infty_c (M)$,
for which the inequalities in \eqref{eqn-F-G-X-Y-ineqs-defn-pb4} are
replaced by equalities on some open neighborhoods of the sets $X_0,
X_1, Y_0, Y_1$.

\bigskip
\noindent
{\bf If it is clear from the context what $X_0,X_1,Y_0,Y_1$ are meant, we
will omit the corresponding indices and sets in the notation for
$pb^+_4$, $\cF$, $\cF'$.}

The number $1/pb_4^{+} (X_0,X_1,Y_0,Y_1)$ has the following dynamical
interpretation \cite{EP-tetragons} (cf. \cite{BEP}):

\smallskip
\noindent
Consider the set $\goS$ of complete Hamiltonians $G:M\to\R$
such that
$$\Delta_G:= \min_{Y_1}\ G - \max_{Y_0}\ G >0.$$
For each such $G$ define $T_G\in (0,+\infty]$ as the supremum of all
$t>0$ such that there is no trajectory of the Hamiltonian flow of $G$
of time-length
$\displaystyle \leq t/\Delta_G$ from $X_0$ to $X_1$ (such
a trajectory is called a {\it chord of $G$}).
We recall:
\begin{thm}[Theorem 1.11 in \cite{EP-tetragons}]
\label{thm-pb4-dynamical-meaning}
\[
\sup\limits_{G\in \goS} T_G = 1/pb_4^{+} (X_0,X_1,Y_0,Y_1).
\]
\end{thm}
Thus, if there exists
a complete Hamiltonian $G:M\to\R$ with $\min_{Y_1} G - \max_{Y_0} G >0$
that has no chords from $X_0$ to $X_1$ then $1/pb_4^{+} (X_0,X_1,Y_0,Y_1)=+\infty$.

\subsection{The invariant $\bp_L$ for general Lagrangian submanifolds}
\label{subsec-defn-of-pbL-in-gen-case}

The $pb_4^+$-invariant of admissible quadruples can be used to define
a symplectic invariant of Lagrangian submanifolds fibered over the
circle in the following way.

Let $L \subset (M,\omega)$ be a closed, connected Lagrangian submanifold
admitting a fibration over $\SP^1$. Let $\mathcal{I}'(L)$ denote the
set of smooth fibrations $L \to \SP^1$. The right action of $\Diff_0(L)$
defines an equivalence relation on $\cI'(L)$ and we denote the resulting
quotient set by $\cI(L)$. Now slice $\SP^1$ into 4 consecutive closed arcs
$\gamma_1$, $\gamma_2$, $\gamma_3$, $\gamma_4$ {\it in the
counterclockwise order}. For a chosen smooth fibration $f: L \to \SP^1$
representing a class $[f] \in \cI(L)$ we define
\[
X_0: = f^{-1} (\gamma_1),\;\; X_1: = f^{-1} (\gamma_3),\;\;
Y_0: = f^{-1} (\gamma_4),\;\; Y_1: = f^{-1} (\gamma_2).
\]
Roughly speaking, we slice $L$ into four parts along cuts parallel to a
fiber of $f$ -- see Figure~\ref{figure-torus}.

\begin{figure}[h]
    \centering
    \includegraphics[scale=0.64]{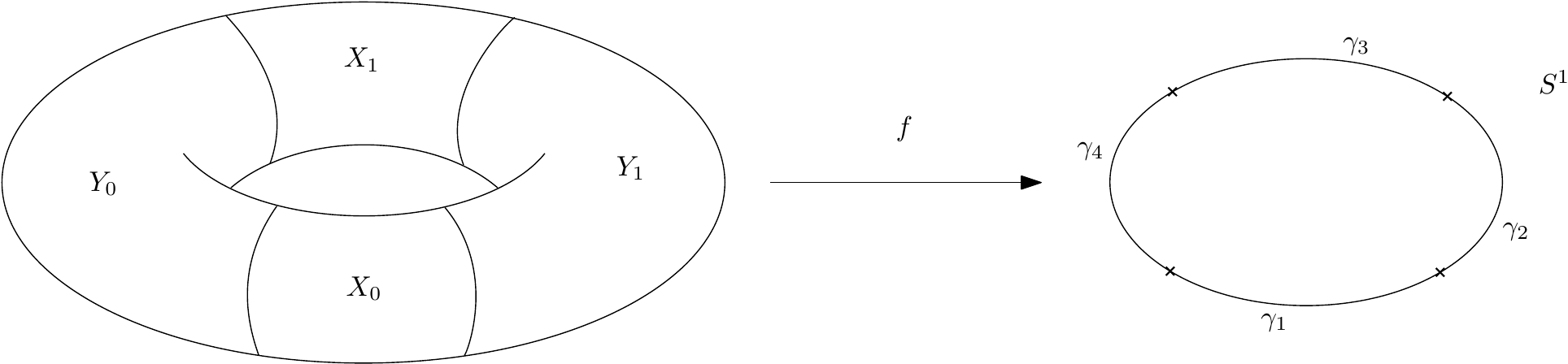}
    \caption{The four sets $X_0, X_1, Y_0, Y_1$ arising from a fibration $f$}
    \label{figure-torus}
\end{figure}

One easily sees that $X_0,X_1,Y_0,Y_1 \subset M$ is an admissible quadruple
and $X_0 \cup X_1 \cup Y_0 \cup Y_1 = L$. We call such a quadruple {\it an
admissible quadruple associated to $f$}. We set
\[
\bp_L ([f]) := 1/pb^+_4 (X_0,X_1,Y_0,Y_1).
\]
If $pb^+_4 (X_0,X_1,Y_0,Y_1)=0$, we set $\bp_L ([f]):=+\infty$.
Thus, $\bp_L ([f])$ takes values in $(0,+\infty]$.

Equivalently, $\bp_L ([f])$ can be described as the infimum of all $T>0$ such
that for {\it any} complete Hamiltonian $G: M\to\R$ satisfying $G|_{Y_0} \leq 0$,
$G|_{Y_1}\geq 1$, there exists a chord of $G$ from $X_0$ to $X_1$ of time-length
$\leq T$. If no such $T$ exists, we set $\bp_L ([f]) := +\infty$. In applications
we will often prove a lower bound $T \leq \bp_L ([f])$ by constructing
for any $\varepsilon > 0$ a complete Hamiltonian $G$ that satifies $G|_{Y_0} \leq 0$,
$G|_{Y_1}\geq 1$ and has no chords from $X_0$ to $X_1$ of time-length
$< T - \varepsilon$.

\begin{rem}
{\rm The letters in the notation $\bp$ stand for the ``Poisson bracket" and their
inverse order ({\it ``b"} before {\it ``p"}) reflects the fact that in the
definition of $\bp$ we take the inverse of the maximum of the Poisson bracket.}
\end{rem}

\begin{prop}
\label{prop-defn-of-pb+L-is-correct}
$\bp_L ([f])$ is well-defined -- i.e. it does not depend on the choice of
representative of $[f] \in \cI(L)$ and the choice of division of $\SP^1$ into
4 arcs.
\end{prop}

\begin{proof}
Let $f: L \to \SP^1$ be a smooth fibration. We first show that
for $\phi \in \Diff_0(L)$  the $\bp_L$-invariants of $f$ and
$f \circ \phi$ are equal.

Note that any $\phi\in \Diff_0(L)$ is a time-1 flow of a
time-dependent vector field on $L$. Denote this vector field
on $L$ by $X_t$ and the corresponding flow by $\phi_t$. Let
$W \cong U \subset T^*L$ be a Weinstein neighborhood of $L$ in
$M$. Identify $L$ with the 0-section in $U$. We can extend $X_t$
in $U$ to a Hamiltonian vector field as follows: in canonical
coordinates $(q,p)$ we define the Hamiltonian $H_t: U \to
\mathbb{R}$ by setting $H_t(q, p) = p( X_t(q))$.  A short
calculation reveals that the Hamiltonian vector field $X_{H_t}
\equiv X_t$ on the zero-section. By multiplying $H_t$ with a
suitable cut-off function one obtains a time-dependent Hamiltonian
with compact support in $W$ such that the induced flow on $L$
coincides with the flow of $\phi_t$. Therefore, by invariance
of $pb_4^+$ under symplectomorphism, $\bp_L$ does not depend on
the choice of representative of a class in $\mathcal{I}(L)$.

We now show independence of the choice of four arcs $\gamma_i$.
Choose four other arcs $\gamma_i'$ in the same fashion. Note that
there exists a $\varphi \in \Diff_0(\SP^1)$ such that $\varphi
(\gamma_i') = \gamma_i$. Using the notation of
Section~\ref{subsec-defn-of-pbL-in-gen-case}, this implies that the image
of $\varphi \circ f$ and $f$ under $\varrho'$ lie in the same
path-connected component of $\cN'(L)$. Using
Proposition~\ref{prop-isom-fibrations-isotopies} we see that
$\varphi \circ f$ and $f$ represent the same equivalence class
in $\cI(L)$. Hence our definition is independent of choice of
arcs.
\end{proof}

Clearly, the resulting function $\bp_L : \cI(L) \to (0,+\infty]$ is a
symplectic invariant of Lagrangian submanifolds fibered over the circle.
We call this function the {\it Poisson-bracket invariant of $L$}.

We now relate $\cI(L)$ to the integral cohomology of $L$.

Consider the set $\cN'(L)$ of all non-singular (that is, non-vanishing)
closed 1-forms on $L$ representing non-zero integral classes in $H^1 (L)$.
It follows easily from Moser's method \cite{Moser}
that the path-connected components of $\cN'(L)$ are exactly the orbits
of the natural $\Diff_0 (L)$-action on $\cN'(L)$. Let $\cN(L) =
\cN'(L)/\Diff_0 (L)$ be the set of the path-connected components of $\cN'(L)$.

Define a map $\varrho': \cI'(L) \to \cN'(L)$ as follows. Given an element
$f \in \cI'(L)$, i.e. a fibration $f: L\to \SP^1$, let $\varrho' (f)$ be
the non-singular 1-form $f^* d\theta$, where $d\theta$ is the standard angle
1-form on $\SP^1$. Clearly, $\varrho': \cI'(L) \to \cN'(L)$ induces a map
$\varrho: \cI(L) \to \cN(L)$.

\begin{prop}\label{prop-isom-fibrations-isotopies}
The map $\varrho: \cI(L) \to \cN(L)$ is invertible.
\end{prop}

\begin{proof}
Define a map $\rho': \cN'(L) \to \cI'(L)$ as follows. Fix a point $x_0\in L$.
Given a non-singular form $\alpha\in \cN^\prime(L)$, let $\rho' (\alpha)$ be
the map $L\to \SP^1 = \R/\Z$ that sends each $x \in L$ to $\int_{x_0}^x
\alpha$ mod 1. Here the integral is taken along any smooth path from $x_0$
to $x$ in $L$ (recall that $L$ is assumed to be connected); a different choice
of path changes the integral by an integral value, since the cohomology class
of $\alpha$ is integral. One easily checks that $\rho' (\alpha)$ is a fibration
of $L$ over $\SP^1$ and thus $\rho'$ is well-defined. Clearly, $\rho'$ induces
a map $\rho: \cN(L) \to \cI(L)$.

Note that $\varrho'\circ\rho' = Id$ and for any fibration $f \in \cI'(L)$ the
fibration $\rho'\circ \varrho' (f)$ lies in the $\Diff_0 (L)$-orbit of $f$.
This shows that $\rho=\varrho^{-1}$.
\end{proof}

Thus $\bp_L$ is also defined as a function $\bp_L: \cN(L) \to (0,+\infty]$.

Note that a path in $\cN^\prime(L)$ has to lie in the same cohomology
class. This defines a map $\Upsilon: \cN(L) \to H^1 (L)\setminus 0$.
For general $L$ the map $\Upsilon$ does not have to be either surjective or
injective -- to check whether it is surjective for a particular $L$ is,
in general, a very non-trivial task, see \cite{Farrell, Kinsey, Latour, Thurston}.
However, if $\dim L \leq 3$, then $\Upsilon$ is injective -- see e.g.
\cite{Laudenbach-Blank}.
Using $\Upsilon: \cN(L) \to H^1 (L)\setminus 0$ we
define a version $\gobp_L$ of the invariant on the image of $\Upsilon$ in
$H^1 (L)\setminus 0$: for $\alpha \in \mathrm{im}(\Upsilon) \subset H^1
(L)\setminus 0$ we set
\[
\gobp_L (\alpha):=\sup\limits_{A \in \Upsilon^{-1} (\alpha)} \bp_L (A),
\]
and we extend $\gobp_L$ to $0$ via $\gobp_L (0) := +\infty$.

\begin{thm}
\label{thm-gopbL-vs-dfL}
Let $L\subset (M,\omega)$ be a closed Lagrangian submanifold admitting
fibrations over $\SP^1$. Then for all $\alpha \in \mathrm{im}(\Upsilon)\cup \{ 0\}
\subset H^1 (L)$ we have
\[
\gobp_L (\alpha)\leq \df_L (\alpha).
\]
\end{thm}

\begin{proof}
For $\alpha =0$ we have $\gobp_L (\alpha) = \df_L (\alpha) = +\infty$
verifying the claim.

Let us assume $\alpha\neq 0$. We will prove the inequality $\gobp_L
(\alpha)\leq \df_L (\alpha)$ by a method developed in \cite{BEP}
(cf. \cite{EP-tetragons}). Namely, given $\alpha\in H^1 (L)$, consider
any smooth fibration $f: L \to \SP^1$ such that $\Upsilon\circ \varrho
([f]) = \alpha$ (see Section~\ref{subsec-defn-of-pbL-in-gen-case}
for the notation). Let $X_0,X_1,Y_0,Y_1$ be an admissible quadruple
used in the definition of $\bp_L ([f])= 1/pb_4^+ (X_0,X_1,Y_0,Y_1)$.
Let $(F,G) \in \cF'(X_0,X_1,Y_0,Y_1)$. Note that, since $L =
X_0\cup X_1\cup Y_0\cup Y_1$ and since $F$ is constant on some
neighborhoods of $X_0$ and $X_1$ and $G$ is constant on some
neighborhoods of $Y_0$ and $Y_1$,
\begin{equation}
\label{eqn-dF-wedge-dG-vanishes-on-a-nbhd-of-L}
dF\wedge dG \equiv 0\ \textrm{ on a neighborhood of}\ L,
\end{equation}
and thus $FdG|_L$ is a closed 1-form on $L$. An easy direct
computation shows that
\begin{equation}
\label{eqn-the-class-of-FdG-equals-minus-alpha}
[F  dG|_L] = -\alpha\in H^1 (L).
\end{equation}
Consider the deformation
\[
\omega_t := \omega + t dF\wedge dG, \ \ t\in\R_{\geq 0}.
\]
A direct calculation shows that
\[
dF\wedge dG\wedge \omega^{n-1} = -\frac{1}{n}\{ F,G\} \omega^n,
\]
and thus
\[
\omega_t^n = (1-t \{ F,G\})\omega^n.
\]
Thus, $\omega_t$ is symplectic for any $\displaystyle t \in
I_{(F,G)}$, where
\[
I_{(F,G)} :=\bigg[0, \frac{1}{\max\limits_M \ \{ F,G\} } \bigg).
\]

Fix an arbitrary $t\in I_{(F,G)}$. Since $F,G$ are compactly
supported, the form $\omega$ can be mapped (using Moser's method
\cite{Moser}) to $\omega_t$ by a compactly supported isotopy
$\vartheta_t: (M, \omega_t)\to (M, \omega)$. Since, by
\eqref{eqn-dF-wedge-dG-vanishes-on-a-nbhd-of-L}, $L$ is Lagrangian
with respect to $\omega_t$, we get that $L_t : = \vartheta_t (L)$ is
a Lagrangian submanifold of $(M,\omega)$ Lagrangian isotopic to $L$.

Using \eqref{eqn-the-class-of-FdG-equals-minus-alpha} we readily see
that, under the identification $H^2 (M,L;\R)\cong H^2 (M,L_t;\R)$
induced by the isotopy, the class $[\omega]_{L_t}$ is identified with
$\omega - t\partial\alpha$. Since this is true for all $t\in I_{(F,G)}$
we obtain
\[
\frac{1}{\max\limits_M \ \{ F,G\} } \leq \df_L (\alpha).
\]
Now the latter inequality holds for all $(F,G) \in \cF'(X_0,X_1,Y_0,Y_1)$,
this gives us
\[
\bp_L ([f]):=  1/pb_4^+ (X_0,X_1,Y_0,Y_1)\leq \df_L (\alpha).
\]
This is true for any $[f]$ such that $\Upsilon\circ \varrho ([f]) =
\alpha$, hence
\[
\gobp_L (\alpha)\leq \df_L (\alpha).
\]
\end{proof}

\begin{question}
\label{question-pbL-vs-def}
Do there exist $L$ and $\alpha \in \mathrm{im}(\Upsilon) \subset
H^1 (L)\setminus 0$ for which $\gobp_L (\alpha)\neq \df_L (\alpha)$?
\end{question}

\subsection{The invariant $\bp_L$ for Lagrangian tori}
\label{subsec-defn-of-bpL-for-Lagr-tori}

Let us now assume that $L$ is diffeomorphic to a torus $\T^n$. In
this case the map $\Upsilon$ is clearly surjective. If $n \leq 3$, then,
as mentioned above, $\Upsilon$ is injective and hence bijective. Thus
$\bp_L = \gobp_L$ coincide and by abuse of notation we write the
invariant
\[
\bp_L: H^1 (L) \longrightarrow \R
\]
in these cases (we extend $\bp_L$ to 0 via $\bp_L(0) := +\infty$).

For a torus of dimension strictly greater than 3 the map $\Upsilon$ may
not be injective -- for instance, $\Upsilon$ is known to be not injective
if $n > 5$, see \cite{Laudenbach, Sikorav}. In any case, each isotopy
class of diffeomorphisms $s: \T^n\to L$ defines a right inverse $\Psi_s:
H^1 (L)\setminus 0\to \cN(L)$ of $\Upsilon$ as follows: given $a\in
H^1(L)\setminus 0$, represent $s^* a\in H^1 (\T^n) \setminus 0$ by a
linear form and let $\Psi_s(a)$ be the path-connected component of
$\cN^\prime(L)$ containing the pull-back of this linear form under
$s^{-1}$. Thus for any isotopy class of diffeomorphisms $s: \T^n\to L$
we can define
\[
\bp_L^s := \bp_L \circ \Psi_s : H^1 (L) \longrightarrow \R,
\]
where we again extend $\bp_L$ to 0 via $\bp_L^s(0) := +\infty$. Clearly,
\begin{equation}
\label{eqn-pbs-gopb}
\bp_L^s \leq \gobp_L.
\end{equation}
In case the class of parametrizations $s: \T^n\to L$ is clear we
sometimes write $\bp_L$ by abuse of notation. For instance, if $L$
is a regular fiber of a Hamiltonian $\T^n$-action there is an
obvious preferred isotopy class of diffeomorphisms $\T^n \to L$.

Theorem~\ref{thm-gopbL-vs-dfL} and \eqref{eqn-pbs-gopb} yield
\begin{equation}
\label{eqn-pbL-dfL}
\bp_L \leq \df_L.
\end{equation}

\begin{rem}
\label{rem-bp-of-multiple-classes-defn}
{\rm
The discussion above shows that for a Lagrangian torus $L$, a
cohomology class $\alpha \in H^1 (L)\setminus 0$ and an isotopy
class $s$ of diffeomorphisms $\T^n\to L$
we have
\[
\bp_L^s (\alpha) = 1/pb^+_4 (X_0,X_1,Y_0,Y_1).
\]
Here $(X_0, X_1, Y_0, Y_1)$ is an admissible quadruple associated
to a fibration $f_\alpha: L\to\SP^1$ such that the $\Diff_0 (L)$-orbit
of the 1-form $f_\alpha^* d\theta$ on $L$ (whose cohomology class is
$\alpha$) is $\Psi_s (\alpha)\in \cN(L)$. By the same token,
\[
\bp_L^s (-\alpha) = 1/pb^+_4 (X_0,X_1,Y_1,Y_0)
= 1/pb^+_4 (X_1,X_0,Y_0,Y_1).
\]
Indeed, $f_{-\alpha}$ can be constructed by composing $f_\alpha$
with an orientation-rever\-sing diffeomorphism of $\SP^1$.

Now assume $k\in\N$ and $\alpha\in H^1 (L)$ is a primitive class
(that is, a class which is not a positive integral multiple of
another class in $H^1 (L)$). Then an admissible quadruple
$(X_0, X_1, Y_0, Y_1)$ associated to $f_{k\alpha}$ can be described
in terms of $f_\alpha$. Namely, divide $\SP^1$ in consecutive closed
arcs $\gamma_1, \ldots, \gamma_{4k}$ in the counterclockwise order.
Set
\[
X_0 := f_\alpha^{-1} \bigg( \bigcup\limits_{i\equiv 1\modd 4} \gamma_i \bigg),
\; X_1 := f_\alpha^{-1} \bigg( \bigcup\limits_{i\equiv 3\modd 4} \gamma_i\bigg),
\]
\[
Y_0:=f_\alpha^{-1} \bigg( \bigcup\limits_{i\equiv 0\modd 4} \gamma_i\bigg),
\; Y_1:=f_\alpha^{-1} \bigg( \bigcup\limits_{i\equiv 2 \modd 4} \gamma_i\bigg).
\]
Thus, if $k=1$, then the sets $X_0,X_1,Y_0,Y_1$ are diffeomorphic to
$\T^{n-1}\times [0,1]$, while if $k>1$ the sets are the unions of the
same number of disjoint copies of $\T^{n-1}\times [0,1]$.
}
\end{rem}


\subsection{General properties of $\bp_L$ and $\df_L$}
\label{subsec-pb4-generalities}

\noindent We list basic properties of $\bp_L$ and $\df_L$ that
will be used further in the paper.

\bigskip
\noindent{\sc Homogeneity of $\df_L$:}
\medskip

$\df_L$ is positively homogeneous of degree $-1$,
\[
\df_L(c \alpha) = \df_L(\alpha)/c
\]
for any $c > 0$ and $\alpha \in H^1(L;\R)$.

\begin{question}
Is there an inequality/equality between $\bp_L (k\alpha)$ and
$k\bp_L (\alpha)$ in case $L$ is a Lagrangian torus and $k\in\N$?

\end{question}

\bigskip
\noindent{\sc Semi-continuity of $\df_L$:}
\medskip

Let $L_j \subset (M, \omega)$, $j \in \N$, be a sequence of Lagrangian
submanifolds Lagrangian isotopic to $L$ and converging to $L \subset
(M, \omega)$ in the $C^1$-topology. Then
\begin{equation}
\label{eqn-def-semi-cont}
\df_{L}(\alpha) \leq \liminf_{j\to +\infty} \df_{L_j}(\alpha)\ \textrm{  for any}\ \alpha\in H^1(L; \R).
\end{equation}
(Here we use the canonical isomorphism $H^1(L_j; \R) \cong H^1(L; \R)$).
The inequality follows from a parametric version of the Weinstein neighborhood
theorem.

\medskip
\medskip
\noindent
Now let us consider the general properties of $pb_4^+$. We will use
the following notation: if $U\subset M$ (possibly $U=M$) is an open set
containing an admissible quadruple $X_0,X_1,Y_0,Y_1$  we will denote by
$pb_4^{U,+}(X_0,X_1,Y_0,Y_1)$ the Poisson bracket invariant defined
using functions supported in $U$.

The following properties of $pb_4^+$ and $\bp_L$  follow easily from the
definitions.

\bigskip
\noindent{\sc Monotonicity of $pb_4^+$:}
\medskip

Assume $U\subset W$ are opens
sets in $M$ and $X'_0, X'_1, Y_0', Y'_1\subset U\subset W$
is an admissible quadruple. Let $X_0,X_1,Y_0,Y_1$ be another
admissible quadruple such that $X_0\subset X'_0, X_1 \subset
X'_1, Y_0\subset Y'_0, Y_1 \subset Y'_1$. Then
\begin{eqnarray}
\label{eqn-pb4+-monotonicity}
pb_4^{W,+} (X_0,X_1,Y_0,Y_1) \leq pb_4^{U,+} (X'_0, X'_1,
Y'_0, Y'_1).
\end{eqnarray}

\bigskip
\noindent{\sc Semi-continuity of $pb_4^+$ and $\bp_L$:}
\medskip

Suppose that a sequence
$X_0^{(j)}, X_1^{(j)}, Y_0^{(j)}, Y_1^{(j)}$, $j\in\N$,
of ordered collections converges (in the sense of the
Hausdorff distance between sets) to a collection $X_0,
X_1, Y_0, Y_1$. Then
\[
\limsup_{j\to +\infty} pb_4^+ (X_0^{(j)}, X_1^{(j)},
Y_0^{(j)}, Y_1^{(j)}) \leq pb_4^+ (X_0,X_1,Y_0,Y_1).
\]
Accordingly,
if $L_j \subset (M, \omega)$, $j \in \N$, is a sequence of Lagrangian
submanifolds Lagrangian isotopic to $L$ and converging to $L \subset
(M, \omega)$ in the $C^1$-topology. Then
\begin{equation}
\label{eqn-Hausdorff-convergence}
\bp_L (\alpha) \leq \liminf_{j\to +\infty} \bp_{L_j}(\alpha)\ \textrm{for any}\ \alpha\in H^1(L; \R).
\end{equation}
(Here we use the canonical isomorphism $H^1(L_j; \R) \cong H^1(L; \R)$).

\bigskip
\noindent{\sc Behavior of $pb_4^+$ and $\bp_L$ under products:}
\medskip

Suppose that $M$ and $N$ are connected symplectic manifolds. Equip
$M\times N$ with the product symplectic form. Let $K
\subset N$ be a compact subset. Then for every collection
$X_0, X_1,Y_0,Y_1$ of compact subsets of $M$
\begin{equation}\label{eqn-pb-products}
pb_4^{M\times N,+} (X_0 \times K, X_1 \times K, Y_0
\times K, Y_1 \times K)\leq pb_4^{M,+} (X_0, X_1,Y_0,Y_1).
\end{equation}

The following product property of $\bp_L$ follows immediately
from \eqref{eqn-pb-products}:

\begin{prop}
\label{prop-product-property-of-pbL} Assume $L_i \subset (M_i,
\omega_i)$, $i=1,2$, are Lagrangian tori and $\alpha\in H^1(L_1)$.
Consider the Lagrangian submanifold $L_1\times L_2 \subset (M_1
\times M_2,\omega_1\oplus\omega_2)$ and the cohomology class $\alpha
\times g\in H^1 (L_1\times L_2)$, where $g$ is a generator of $H^0
(L_2)$. Let $s_i: \T^{n_i} \to L_i$ be two isotopy classes of
diffeomorphisms. Then
\[
\bp^{s_1 \times s_2}_{L_1\times L_2} (\alpha\times g) \geq \bp^{s_1
}_{L_1} (\alpha).
\]\Qed
\end{prop}

\bigskip
\noindent{\sc Behavior of $pb_4^+$ under symplectic reduction:}
\medskip

The following property of $pb_4^+$ did not appear in \cite{BEP},
\cite{EP-tetragons}, but is proved similarly to \eqref{eqn-pb-products}.

Namely, let $(M,\omega)$ be a connected, not necessarily closed,
symplectic manifold. Let $\Sigma \subset (M,\omega)$ be a coisotropic
submanifold. We do not assume that $\Sigma$ is a closed subset of $M$.
Assume that the characteristic
foliation of $\Sigma$ defines a {\it proper} fibration $\pi: \Sigma \to (N,
\eta)$ over a (not necessarily closed) symplectic manifold $(N,\eta)$.

Let $X_0,X_1,Y_0,Y_1\subset N$ be an admissible quadruple (in particular,
the sets are compact). Assume
\[
\hX_0\subset \pi^{-1} (X_0),\ \hX_1\subset \pi^{-1} (X_1),
\ \hY_0\subset \pi^{-1} (Y_0), \ \hY_1\subset \pi^{-1} (Y_1)
\]
are some compact sets in $M$. Then $\hX_0,\hX_1,\hY_0,\hY_1\subset M$ is
an admissible quadruple. Moreover,
\begin{equation}\label{eqn-pb-sympl-reduction}
pb_4^{M,+} (\hX_0,\hX_1,\hY_0,\hY_1)\leq pb_4^{N,+} (X_0, X_1,Y_0,Y_1).
\end{equation}

Indeed, let $F,G$ be functions on $N$ such that $(F,G)\in \cF_N
(X_0,X_1,Y_0,Y_1)$. Consider the functions $F\circ \pi, G\circ \pi$ on
$\Sigma$. Since $\pi:\Sigma\to N$ is proper and $F,G$ are compactly
supported, so are $F\circ\pi, G\circ \pi$. Now cut off the functions
$F\circ \pi, G\circ \pi$ in the isotropic direction normal to $\Sigma$.
As a result we get functions $\chi(r^2) (F\circ\pi), \chi (r^2)
(G\circ\pi)$ with compact support lying in a tubular neighborhood
$U$ of $\Sigma$ -- here $r$ is a radial function on the isotropic
normal bundle to $\Sigma$ with respect to a Riemannian
metric on that bundle and $\chi: \R\to [0,1]$ is a smooth function
supported near $0$ and satisfying $\chi (0)=1$. Extend the two
functions from $U$ to $M$ by zero and denote the resulting functions
on $M$ by $\hF,\hG$. One easily checks that $(\hF,\hG)\in \cF_M
(\hX_0,\hX_1,\hY_0,\hY_1)$ and $\max_M\ \{ \hF,\hG\} = \max_N \
\{F,G\}$, which yields \eqref{eqn-pb-sympl-reduction}.

In the case when $\Sigma$ is a fiber of the moment map of a
Hamiltonian action of a Lie group $H$ on $M$ the reduced space
$N = \Sigma/H$ may not be a smooth symplectic manifold but a
symplectic orbifold, which brings us to the following discussion.

\bigskip
\noindent{\sc The Poisson bracket invariant for symplectic orbifolds:}
\medskip

Recall (see e.g. \cite{Lerman-Tolman-TAMS97}) that the notions of
smooth functions, vector fields and differential forms can be
extended to orbifolds. In particular, there is a well-defined notion
of an orbifold symplectic form; an orbifold equipped with such a
form is called a {\it symplectic orbifold}. Accordingly, there is a
notion of the Poisson bracket of two smooth functions on a
symplectic orbifold and the definition of $pb_4^+$ can be carried
over literally to symplectic orbifolds.

It is easy to check that the proof of \eqref{eqn-pb-sympl-reduction}
goes through in the case when the reduced symplectic space $N$ is an
orbifold.


\section{Lagrangian tori in symplectic surfaces -- proofs}
\label{sec-pf-thm-pbL-surfaces}

\bigskip
\noindent
{\bf Proof of Theorem~\ref{thm-pbL-surfaces}.}

Let us prove that $\df_L (ke) \leq A_+/k$. (The inequality $\df_L (-ke) \leq A_-/k$ is proved
in the same way). If $A_+ = \infty$, the inequality is trivial, so let us assume that $A_+<\infty$.

Fix $\epsilon>0$.
Let $\psi=\{ \psi_t: L\to M\}_{0\leq t\leq T}$ be a Lagrangian isotopy  of $L$
such that $\Flux (\psi)_t=-tke$ for all $0\leq t\leq T$.
By definition of $\df_L (ke)$ and since $\epsilon$ was chosen arbitrarily, it suffices to show that
$T\leq A_+ / k +\epsilon$.

There exists a compact surface $K$ (possibly with boundary) which lies in
$M$ and
contains the union $\cup_{0\leq t\leq T} \psi_t (L)$ of all the
Lagrangian submanifolds appearing in the isotopy. Cap off the boundary components of $K$
lying in $M_+$ (if they exist) by disks and extend the symplectic form from $K$ over the disks
so that the total area of the disks is smaller than $\epsilon$. Denote the resulting compact
symplectic surface by $\hK$ and the symplectic form on it by $\homega$.

Note that $L$ is a Lagrangian submanifold of $(\hK,\homega)$, $\psi$ defines a Lagrangian
isotopy of $L$ in $(\hK,\homega)$ and the Lagrangian flux of the latter Lagrangian isotopy in
$\hK$ is the same as that of the original Lagrangian isotopy in $M$. By our construction, $L$
bounds in $\hK$ a domain {\it without boundary} of area $A < A_+ + \epsilon$. Therefore
$\partial (ke) = k [\homega]_L/A$ and $(\hK,\homega)$ does not admit weakly exact Lagrangian
submanifolds in the Lagrangian isotopy class of $L$ (because any curve in $\hK$ isotopic to
$L$ bounds a domain of positive area). Thus we can apply
Theorem~\ref{thm-partial-alpha-proportional-to-omega-upp-bound-on-bpL-alpha} and get that
$\df_L (ke)$, for $L$ viewed as a Lagrangian submanifold of $(\hK,\homega)$, is not bigger
than $A / k$. On the other hand, by definition, $T\leq \df_L (ke)$. Thus, $T\leq A/k < A_+/k
+ \epsilon$. Since this holds for any Lagrangian isotopy $\psi$ as above and any $\epsilon >0$, by the definition
of $\df_L (ke)$ for $L$ viewed as a Lagrangian submanifold of $(M,\omega)$, we get that
$\df_L (ke) \leq A_+/k$.

Since, by Theorem~\ref{thm-gopbL-vs-dfL}, $\bp_L (\pm ke)
\leq \df_L (\pm ke)$,
it remains to prove that $\bp_L (\pm ke) \geq A_\pm/k$.

We first consider a model situation. Let $\varepsilon, A > 0$ and
denote by $(x, y) \in \R^2$ the coordinates and by $\pi: \R^2 \to \R$
the projection onto the $x$-axis, $\pi(x, y) = x$ . Define
\[
Q(A) := [0, A] \times [0,1] \subset Q_{\varepsilon}(A) := (-\varepsilon,
A + \varepsilon) \times (-\varepsilon,1 + \varepsilon) \subset \R^2.
\]
Label three sides of $Q(A)$ as follows:
\[
\hX_0 = [0, A] \times \{ 0 \}, \ \
\hY_0 = \{ 0 \} \times [0,1], \ \
\hY_1 = \{ A \} \times [0,1].
\]
For the remaining side $[0, A] \times \{ 1 \}$ we choose a partition into
$4k - 3$ closed intervals $\gamma_1, \ldots, \gamma_{4k - 3}$,
ordered from right to left, such that for $i \equiv 0, 2, 3 \, \modd 4$
the intervals $\gamma_i$ have length $\varepsilon$, the interval
$\gamma_1$ has length $\tfrac{A}{k}$ and the remaining intervals
have length $\tfrac{A}{k} - 3\varepsilon$. Set
\begin{align*}
X_0 & := \hX_0 \cup \bigcup\limits_{i\equiv 0\modd 4} \gamma_i,
  & X_1 & :=  \bigcup\limits_{i \equiv 1 \modd 4} \gamma_i, \\
Y_0 & := \hY_0 \cup \bigcup\limits_{i\equiv 2\modd 4} \gamma_i,
  & Y_1 & :=  \hY_1 \cup \bigcup\limits_{i\equiv 0 \modd 4} \gamma_i.
\end{align*}
Now choose a piecewise-linear function $G_{\varepsilon}: Q_{\varepsilon}(A)
\to [0,1]$ that satisfies:
\begin{enumerate}
\item[(i)] $G_{\varepsilon}|_{Y_0} = 0$, $G_{\varepsilon}|_{Y_1} = 1$,
\item[(ii)] $G_{\varepsilon}$ only depends on $x$ in $Q(A)$,
\item[(iii)] $G_{\varepsilon}$ has compact support.
\end{enumerate}
One can choose $G_{\varepsilon}$ to satisfy $\partial_x G_{\varepsilon}
\leq 1 / (\tfrac{A}{k} - 3 \varepsilon)$ on $\pi(X_1) \times [0,1]
\subset Q(A)$, see for example Figure~\ref{figure-surface-upper-bound}.
Therefore there exists a smooth approximation $G$
that satisfies (i), (ii) and (iii) and has slope
$\partial_x G < 1 / (\tfrac{A}{k} - 4 \varepsilon)$ on $\pi(X_1) \times [0,1]$.
This implies that all chords of $G$ from $X_0$ to $X_1$ have time-length
$T >  \tfrac{A}{k} - 4\varepsilon$. Since $\varepsilon > 0$ can be chosen
arbitrarily small, by the dynamical characterization of $\bp_L$ in
Section~\ref{subsec-defn-of-pbL-in-gen-case}, we have
$1 / pb_4^+(X_0, X_1, Y_0, Y_1) \geq \tfrac{A}{k}$.

\begin{figure}[h]
    \centering
    \includegraphics[scale=0.74]{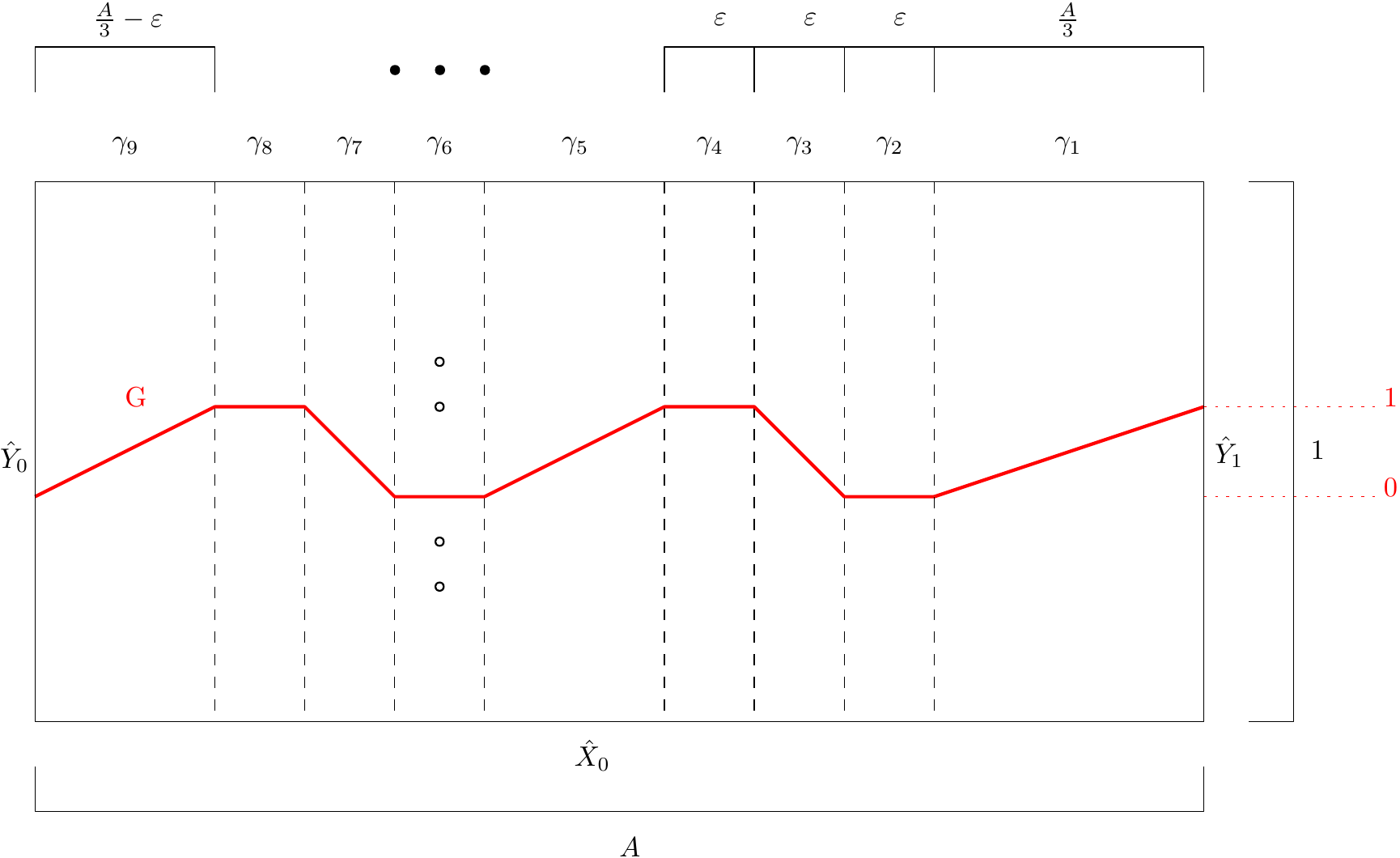}
    \caption{Partition and function in the case $k = 3$}
    \label{figure-surface-upper-bound}
\end{figure}

We now construct a specific neighborhood of $L \subset M$. We first treat the case
where $L$ is a boundary component of $M_+$. Note that the boundary orientation
of $L$ corresponds to $+ e$.

For a small $\delta > 0$ choose $A = A_{+} - \delta$. Let $D_1, \ldots, D_{n} \subset \Int Q(A)$
denote a non-intersecting (possibly empty) finite union of closed
disks of total symplectic area less than $\delta$. Then for a sufficiently small $\varepsilon > 0$
and for an appropriate choice of the disks $D_i$ there exists an open neighborhood $U_{+} \subset \Int M$
of $L$ that can be symplectically identified with $Q_{\varepsilon}(A) \backslash D_1 \cup \ldots \cup
D_{n}$ and so that $L$ (with the orientation corresponding to $+e$) gets mapped arbitrarily close to
$\partial Q(A)$ with the standard boundary orientation. If $\delta > 0$ is chosen small enough, one can
choose the disks $D_i$ to lie in the region $\pi(Y_0) \times [0,1] \subset Q(A)$.
Thus we can extend the function $G$ from $U_+$,
identified with $Q_{\varepsilon}(A) \backslash D_1 \cup \ldots \cup D_{n}$, to the
whole symplectic manifold $M$ by $0$. Then all chords
of $G$ from $X_0$ to $X_1$ have time-length $T >  \tfrac{A}{k} - 4\varepsilon$.
Since $\varepsilon, \delta > 0$ can be chosen arbitrarily small,
together with the semi-continuity and symplectic invariance
properties of $pb_4^+$ (see Section~\ref{subsec-pb4-generalities})
this implies that $\bp_L (ke) \geq A_{+}/ k$.

When $L$ with the orientation $- e$ is a boundary component of $M_-$
a similar construction gives us the lower bound $\bp_L (- ke)
\geq A_{-}/ k$. This completes the proof.
\Qed

\bigskip
\noindent
{\bf Proof of Theorem~\ref{thm-surface-separ-curve}.}

The case $k = \pm 1$ was shown in \cite{Samvelyan}. We generalize this
construction to arbitrary $k \in \Z$.

By Remark~\ref{rem-bp-of-multiple-classes-defn}, we see that an
admissible quadruple associated to $k e$ is given by dividing $L \simeq
\SP^1$ into $4k$ consecutive closed arcs and labeling them
with $X_0, Y_1, X_1, X_1$ following the orientation of $L$ in the
case $k > 0$ and else following the opposite orientation in the case
$k < 0$. Note that, by Proposition~\ref{prop-defn-of-pb+L-is-correct},
$\bp_L (ke)$ does not depend on the choice of subdivision as long as the
order of the $4k$ arcs is preserved.

For $\varepsilon > 0$ consider the symplectic annulus,
\[
(Z_{\varepsilon} = (-\varepsilon, \varepsilon) \times [-1, 1] / \mathord\sim,
\ dx \wedge dy)
\]
where we identify $(x, -1) \sim (x, 1)$. Partition the subset
$(-\varepsilon, \varepsilon) \times \{ 0\}$ into $4k + 1$ consecutive
intervals $\gamma_1, \ldots, \gamma_{4k + 1}$ respecting the
standard orientation on $(-\varepsilon, \varepsilon)$. Set
\[
\hX_0 := \bigcup\limits_{i\equiv 1\modd 4} \gamma_i,
\quad \hX_1 := \bigcup\limits_{i\equiv 3\modd 4} \gamma_i,
\]
\[
\hY_0:= \bigcup\limits_{i\equiv 0\modd 4} \gamma_i,
\quad \hY_1:= \bigcup\limits_{i\equiv 2 \modd 4} \gamma_i.
\]
One can easily construct functions $F, G: Z_{\varepsilon}
\to [0,1]$ such that $(F, G) \in \cF' (\hX_0,\hX_1,\hY_0,\hY_1)$
and $\{ F, G \} \equiv 0$, namely choose $F$ and $G$ to only
depend on the $x$-coordinate.

The Lagrangian $L \subset M$ is not a separating curve. Thus
there exists an embedded loop $\gamma \subset M$ that intersects
$L$ exactly once. Furthermore, one can find a tubular neighborhood $N$
of $\gamma$ that is symplectomorphic to $Z_{\varepsilon}$ for
$\varepsilon > 0$ small enough. The symplectomorphism can be
chosen to map $\gamma$ to $\{0\} \times [-1,1]$ and
$L \cap N$ to $(-\varepsilon, \varepsilon) \times \{ 0 \}$.

We now make a specific choice of admissible quadruple
$X_0,X_1,Y_0,Y_1$ on $L$. Namely choose $X_0 \subset L$ such
that $L \backslash N \subset X_0$ and $X_0 \cap N = \hX_0$,
$X_1 \cap N = \hX_1$, $Y_0 \cap N = \hY_0$ and
$Y_1 \cap N = \hY_1$ (or, if the order requires, use the
reflection of $\hX_0,\hX_1,\hY_0,\hY_1 \subset Z_{\varepsilon}$
along the $y$-axis). Since $F, G$ have compact support in
$Z_{\varepsilon}$, we may pull back the functions to $M$
and obtain functions on $M$ that satisfy
\eqref{eqn-F-G-X-Y-ineqs-defn-pb4}, have compact support and
Poisson-commute. This proves the theorem.
\Qed

\bigskip
\bigskip
\noindent{\sc A lower bound on $\bp_L$ for orbifolds:}
\medskip

The proof of the lower bound for $k \in \N$,
\begin{equation}
\label{eqn-pbL-upper-bound-orbifold}
A_\pm/k \leq \bp_L (\pm k e),
\end{equation}
from the proof of Theorem~\ref{thm-pbL-surfaces} can also be modified for the
orbifold case.

Indeed, the singular points of a 2-dimensional symplectic orbifold $M$ form a discrete
subset of $M$ \cite{Gordon} and one can assume without loss of generality that they
all lie in $M\setminus U_+$ where the function $G$ vanishes identically (and thus can
be extended over the singularities).

\bigskip
\bigskip
\noindent{\sc Symplectic reduction and $\bp_L$:}
\medskip

Let $(M^{2n},\omega)$ be a connected, not necessarily closed,
symplectic manifold. Let $\Sigma^{n+1} \subset (M,\omega)$ be a
smooth coisotropic submanifold. We do not assume that $\Sigma$ is
a closed subset of $M$. Assume that the characteristic foliation
of $\Sigma$ defines a {\it proper} fibration $\pi: \Sigma \to (N,
\eta)$ over a (not necessarily closed) 2-dimensional symplectic
orbifold $(N^2,\eta)$. Let $L\subset (M,\omega)$ be a Lagrangian
torus lying in $\Sigma$ and $s$ an isotopy class of diffeomorphisms
$\T^n \to L$. Assume that $\pi|_L : L\to \Gamma$ is a fiber bundle
over a simple closed curve $\Gamma:=\pi (L)$ lying in the
non-singular part of $N$ and dividing $N$ into two domains of
areas $A_+$ and $A_-$.

Equip $\Gamma$ with an orientation induced by the orientation of
the domain of area $A_+$. The orientation of $\Gamma$ defines a
positive generator $e\in H^1 (\Gamma)\cong\Z$.

Assume that $k\in \N$ and $\alpha\in H^1 (L)$ is a primitive class.
Assume that there exists an orientation-preserving diffeomorphism
$\SP^1\to\Gamma$ that identifies the fiber bundle $\pi|_L : L \to
\Gamma$ with a fibration $f_\alpha: L\to\SP^1$ associated to $s$
and $\alpha$ -- see Remark~\ref{rem-bp-of-multiple-classes-defn}.
Then an admissible quadruple associated to a fibration $f_{k\alpha}$
can be described in terms of $f_\alpha$ as follows (see
Remark~\ref{rem-bp-of-multiple-classes-defn}).

Divide $\SP^1\cong \Gamma$ in $4k$ consecutive closed arcs $\gamma_1,
\ldots, \gamma_{4k}$ numbered in the counterclockwise order. Set
\[
X_0 := \bigcup\limits_{i\equiv 1\modd 4} \gamma_i,
\ X_1 := \bigcup\limits_{i\equiv 3 \modd 4} \gamma_i,
\]
\[
Y_0:=\bigcup\limits_{i\equiv 0\modd 4} \gamma_i,
\ Y_1:=\bigcup\limits_{i\equiv 2\modd 4} \gamma_i,
\]
and furthermore
\[
\hX_0:= \pi^{-1} (X_0)\cap L = f_\alpha^{-1} (X_0),\
\hX_1:= \pi^{-1} (X_1)\cap L = f_\alpha^{-1} (X_1),
\]
\[
\hY_0:= \pi^{-1} (Y_0)\cap L = f_\alpha^{-1} (Y_0),\
\hY_1:= \pi^{-1} (Y_1)\cap L = f_\alpha^{-1} (Y_1).
\]
By Remark~\ref{rem-bp-of-multiple-classes-defn}, $(\hX_0, \hX_1,
\hY_0, \hY_1)$ is an admissible quadruple associated to $f_{k\alpha}$.

Combining this observation with \eqref{eqn-pb4+-monotonicity} and
\eqref{eqn-pb-sympl-reduction} (in the orbifold case), we immediately
get:
\[
\bp_L^s (k\alpha) = 1 / pb^{M,+}_4 (\hX_0, \hX_1, \hY_0, \hY_1)
\geq  1 / pb^{N,+}_4 (X_0, X_1, Y_0, Y_1) =\bp_\Gamma (k e).
\]
In view of
\[
\bp_L^s (-k\alpha) = 1 / pb^{M,+}_4 (\hX_0, \hX_1, \hY_1, \hY_0)
\]
(see Remark~\ref{rem-bp-of-multiple-classes-defn}) and
\eqref{eqn-pbL-upper-bound-orbifold} (in the orbifold case)
this yields the following claim.

\begin{prop}
\label{prop-upper-bound-on-pb-4-plus-for-toric-actions-two-dim-orbifold}
With the setup as above, for any $k \in \N$ we have
\[
\bp^s_L (k \alpha)\geq A_+/k
\]
and
\[
\bp^s_L (-k \alpha)\geq A_-/k.
\]
\Qed
\end{prop}


\section{Toric orbits in symplectic toric manifolds -- proofs}
\label{sec-pf-thm-toric-case-low-bound}

In this section we prove Theorem~\ref{thm-toric-case-low-bound}.

Set $k\alpha:=(m_1,\ldots,m_n)\in (\Z^n)^*=H^1 (\T^n)$, where $\alpha$ is
a primitive class and $k\in\N$.

Complete $\alpha$ to an integral basis $\alpha_1,\ldots,\alpha_{n-1},
\alpha$ of $H^1 (\T^n)=(\Z^n)^*$ and let $\beta_1,\ldots,\beta_n$ be
the dual integral basis of $\Z^n$. Thus, we have a splitting
\[
\Lied \T^n = \Span_\R \{ \alpha_1,\ldots,\alpha_{n-1}\}\oplus
\Span_\R \{ \alpha\}.
\]
Let $\pi_1:  \Lied \T^n\to \Span_\R \{\alpha_1,\ldots,\alpha_{n-1}\}$
and $\pi_2:  \Lied \T^n\to \Span_\R\{ \alpha\}$ be the projections
defined by the splitting.

Consider the $(n-1)$-dimensional subtorus $H\subset \T^n$ whose Lie
algebra is
\[
\Lie H:= \Ker\, \alpha = \Span_\R \{
\beta_1,\ldots,\beta_{n-1}\}\subset \R^n = \Lie \T^n.
\]
The map $\Lied \T^n\to \Lied H$ dual to the inclusion $\Lie H\to\Lie
\T^n$ can be identified with the projection $\pi_1$.

Since $H$ is a subtorus of $\T^n$, there is a Hamiltonian action of
$H$ on $(M,\omega)$ whose moment map $\Phi_H$ can be described as the
composition of $\Phi$ and $\pi_1$. Thus, $L_\bx=\Phi^{-1} (\bx)$ lies
in a fiber $\Sigma'_\bx$ of $\Phi_H$ which is the union of the fibers
$\Phi^{-1}(\by)$ for all $\by\in\Delta$ such that $\pi_1 (\by) = \pi_1
(\bx)$. The set $\Sigma'_\bx$ may be an orbifold, but its smooth part
$\Sigma_\bx\subset\Sigma'_\bx$, which is the union of the fibers
$\Phi^{-1} (\by)$ for $\by \in \Int\Delta$ with $\pi_1 (\by) = \pi_1
(\bx)$, is a smooth coisotropic $(n+1)$-dimensional submanifold of $M$.

The torus $H$ acts on $\Sigma_\bx$ and the reduced space $\Sigma_\bx/H$
is a 2-dimensional symplectic orbifold $N_\bx$. The natural projection
$\pi_H:\Sigma_\bx\to\Sigma_\bx/H$ is proper and its fibers are exactly
the leaves of the characteristic foliation of $\Sigma_\bx$. The
1-dimensional torus $\T^n/H$ acts in a Hamiltonian way on $N_\bx$.
In fact, the torus $\T^n/H$ can be identified with a subtorus of
$\T^n$ whose Lie algebra is $\Span_\R \{ \beta_n\}\subset \Lie\T^n$.
Accordingly, $\Lied (\T^n/H)$ is identified with $\Span_\R \{\alpha\}
\subset \Lied \T^n$. Thus, the moment map $\Phi_{\T^n/H}$ of the $
\T^n/H$-action on $N_\bx$ can be viewed as a map $\Phi_{\T^n/H}:
N_\bx\to \Span_\R \{ \alpha\}$.

A well-known property of Hamiltonian group actions (the so-called
``reduction in stages", see e.g. \cite[Exercise III.12]{Audin-toric-book-2nd-ed})
implies that the orbits of the original $\T^n$-action on $(M,\omega)$
lying in $\Sigma_\bx$ project under $\pi_H: \Sigma_\bx \to N_\bx$ to
the orbits of the $\T^n/H$-action on $N_\bx$ and for any $\by\in\Lied
\T^n \cap \Delta$ such that $\pi_1 (\bx)=\pi_1 (\by)$ we have
\[
\pi_2\circ \Phi  =  \Phi_{\T^n/H} \circ \pi_H \;\;\textrm{on}\;\;
\Phi^{-1}(\by).
\]
Therefore the image of the moment map $\Phi_{\T^n/H}$ can be
identified with the image under $\pi_2$ of the intersection of the
line $\bx - t \alpha$, $t\in\R$, in the affine space $(\R^n)^*=\Lied
\T^n$ with $\Delta= \Image\Phi$.

Another conclusion is that the torus $L_\bx\subset \Sigma_\bx$, which
is an orbit of the $\T^n$-action on $(M,\omega)$, projects under
$\pi_H: \Sigma_\bx\to\Sigma_\bx/H$ to a simple closed curve
$\Gamma\subset N_\bx$ that lies in the non-singular part of the
orbifold $N_\bx$. The curve $\Gamma$ is an orbit of $\T^n/H$ and as
such can be identified with $\SP^1$ -- this identification is unique
up to a rotation of $\SP^1$. Moreover, under this identification the
map
\[
\pi_H |_{L_\bx}: L_\bx\to \Gamma
\]
becomes a fibration $f_\alpha: L_\bx\to\SP^1$ associated to $\alpha$
-- see Remark~\ref{rem-bp-of-multiple-classes-defn}. (Indeed, $f_\alpha$
can be viewed as the projection $L_\bx\to L_\bx/H=\SP^1$).

The symplectic properties of $\Gamma$ inside $N_\bx$ are completely
determined by the relative position of the point $\pi_2
(\bx)=\Phi_{\T^n/H} (\Gamma)$ in the image of $\Phi_{\T^n/H}$ or,
equivalently, by the position of $\bx$ in the intersection of the
line $\bx - t \alpha$, $t\in\R$, with $\Delta$. Recall that $\cI
(\bx,\alpha)$ is defined as the open part of the intersection of the
ray $\bx - t \alpha$, $0<t< +\infty$, with $\Delta$. By a basic
version of the Delzant theorem \cite{Delzant}, the oriented curve
$\Gamma$ (it is oriented as an orbit of $\T^n/H$) is the oriented
boundary of a domain $D$ in $N_\bx$ whose area is the rational
length of $\cI (\bx,\alpha)$, if $\cI (\bx,\alpha)$ is an interval,
or $+\infty$, if $\cI (\bx,\alpha)$ is an infinite ray.

Thus, we can apply
Proposition~\ref{prop-upper-bound-on-pb-4-plus-for-toric-actions-two-dim-orbifold},
which implies that
\[
\bp_{\bx} (k \alpha)\geq l_\bx (\alpha)/k
\]
for $k \in \N$.\Qed


\section{Lagrangian tori in $\C^2$ -- proofs}
\label{sec-lagr-tori-in-C2-pfs}

\subsection{Proof of Theorem~\ref{thm-lagrangian-tori-in-R4}}
\label{subsec-lagr-tori-in-R4-pf}

The original idea of the proof below belongs to E. Opshtein.

Consider a Lagrangian isotopy
\[
\psi = \{ \psi_t: L \to \C^2 \}, \ 0 \leq t \leq T, \
\psi_0 = \iota,
\]
such that
\[
[\omega]^{\psi}_t = [\omega]_L - t \partial\alpha,
\]
with $L \subset \C^2$ and $\alpha$ as in the statement
of the theorem.

For simplicity let us consider a dual picture. Namely,
for $0 \leq t \leq T$ there exists a family of compactly supported
diffeomorphisms $\varphi_t: \C^2 \to \C^2$, $\varphi_0 = \mathrm{Id}$,
such that $\varphi_t(L) = \psi_t(L)$. By pulling back our symplectic
form $\varphi_t^*\omega =: \omega_t$ we may consider a fixed Lagrangian
$L \subset (\C^2, \omega_t)$. We have $H_2(\C^2, L) \simeq \Z\langle A,
B \rangle$ and $\omega_t(A) = \omega(A) - {\sigma} t$ and $\omega_t(B) =
\omega(B) - {\rho} t$.

Let $\cJ_t$ be the space of almost complex structures on $\C^2$ compatible with the symplectic form $\omega_t$.
Given an almost complex structure $J \in \cJ_t$, by a {\it (parameterized) $J$-holomorphic disk} we always
mean a smooth $J$-holomorphic map $u: (\D, \partial \D)\to (\C^2, L)$ (here $\D\subset \C$ is the standard closed
unit disk). By the homology class of a $J$-holomorphic disk $u$ we always mean the relative homology class
$u_*([\D]) \in H_2 (\C^2, L)$. Such a $u$ is called {\it somewhere injective}, if $du (z)\neq 0$ for some $z\in \D$
such that $u^{-1} (u(z))=\{ z\}$.

Given a relative homology class $C\in H_2 (\C^2, L)$ define $\tcM (C, J)$ as the moduli
space of somewhere injective (parametrized) $J$-holomorphic disks in the class $C$. Let $\tcM (C, J)\times \partial\D$
be the moduli space of $J$-holomorphic disks with one marked point on the boundary. $PSL_2 (\R)$ is the group
of biholomorphisms of $\D$ and we consider the quotient
\[
\cM_1 (C, J) := (\tcM (C, J)\times \partial\D) / PSL_2 (\R),
\]
where the action of $PSL_2 (\R)$ is defined as $g\cdot (u,x) = (u\circ g, g^{-1} (x))$. This space comes with
an evaluation map $ev: \cM_1 (C, J) \to L$ given by $ev([u,x]) := u(x)$ and for a chosen point $p \in L$ we define
$\cM_1 (C, J, p) := ev^{-1}(p)$.

Given a family $\{ J_t\}$, $0\leq t\leq T$, of almost complex structures $J_t\in \cJ_t$ and a class
$C\in H_2 (\C^2, L)$, define $\cM_1 (C, \{ J_t\})$ as the set of pairs $(t,D)$, where $0\leq t\leq T$
and $D\in \cM_1 (C, J_t)$.
The set $\cM_1 (C, \{ J_t\},p)\subset \cM_1 (C, \{ J_t\})$ is defined analogously
with $D \in \cM_1 (C, J_t, p)$.

We will show that for $A$ as in the assumption of the theorem
and for a generic family $\{ J_t\}$ the set
$\cM_1 (A, \{ J_t\},p)$ is a smooth compact manifold of dimension 1.

We recall first some general facts about the moduli spaces.

We say that $J \in \cJ_t$, $t\in [0,T]$, is {\it regular}, if for all $C\in H_2 (\C^2,L)$ the space $\cM_1 (C,J)$
is a (transversally cut out) smooth manifold of dimension
\[
\dim \cM_1 (C,J) = \dim L + \mu(C) + 1 - \dim PSL_2 (\R) = \mu(C).
\]
An almost complex structure $J\in\cJ_t$ is called {\it regular with respect to $p\in L$}, if it is regular and, in addition, 
$\cM_1 (C,J,p)$ is a (transversally cut out) smooth manifold of dimension
\[
\dim \cM_1 (C,J,p) = \dim \cM_1 (C, J) - \dim L = \mu (C) - 2.
\]
We will say that a family  $\{ J_t\}$, $0\leq t\leq T$, $J_t\in \cJ_t$, is {\it regular
with respect to $p$} if
\begin{enumerate}
\item[(1)] for any $t\in [0,T]$ the spaces $\cM_1 (C, J_t)$ are empty for all $C$ with $\mu (C)<0$,
\item[(2)] $\cM_1 (C, \{ J_t \}, p)$ is a (transversally cut out) smooth manifold of dimension
\[
\dim \cM_1 (C, \{ J_t \},p) = \dim \cM_1 (C, J, p) +1 = \mu (C) - 1
\]
with boundary $\cM_1 (C, J_0, p) \cup \cM_1 (C, J_T, p)$.
\end{enumerate}
Similarly, given a (regular) path $\gamma (s)$, $0\leq s\leq 1$, in $L$, and $t\in [0,T]$, we say that a family $\{ J_s\}_{0\leq s\leq 1}\subset \cJ_t$ is {\it regular with respect to $\gamma$},
if
\begin{enumerate}
\item[(1')] for any $s\in [0,1]$ the spaces $\cM_1 (C, J_s)$ are empty for all $C$ with $\mu (C)<0$,
\item[(2')] $\cM_1 (C, \{ J_s\}, \gamma) := \cup_{0\leq s\leq 1} \cM_1 (C, J_s, \gamma (s))$
is a smooth manifold of dimension $\mu (C) - 1$ with boundary $\cM_1 (C, J_0, \gamma (0)) \cup \cM_1 (C, J_1, \gamma (1))$.
\end{enumerate}

It follows from standard regularity and transversality arguments (see e.g. \cite{McD-Sal-psh-book},
\cite{Oh-book}, where the arguments are explained in detail for pseudo-holomorphic curves
without boundary) that for any $p\in L$

- a generic $J \in \cJ_t$ is regular and moreover regular with respect to $p$,

- for any $p\in L$ and any $J_0\in \cJ_0$, $J_T \in \cJ_T$ that are regular with respect to $p$, a generic family $\{ J_t \}$, $0\leq t\leq T$, $J_t\in \cJ_t$, connecting $J_0$
and $J_T$ satisfies (2).

- for any $\gamma$ as above, any $t\in [0,T]$ and any $J_0, J_1\in \cJ_t$ regular, respectively, with respect to $\gamma (0)$ and $\gamma (1)$,
a generic family $\{ J_s\}_{0\leq s\leq 1}\subset \cJ_t$ connecting $J_0, J_1$
satisfies (2').

In order to show that condition (1) also holds for a generic family $\{ J_t\}$ note that $L$ is
orientable and, accordingly, the Maslov index of {\sl any} disk with boundary on $L$ is even.
Thus, if $\mu (C)$ is negative, then $\mathrm{virtual} \dim \cM_1 (C, J)\leq -2$ for any $J\in \cJ_t$ and $t\in [0,T]$,
meaning that the existence of {\sl somewhere injective} $J_t$-holomorphic disks of negative
Maslov index is a codimension-2 phenomenon and can be avoided by choosing a generic
1-parametric family $\{ J_t\}$.
This shows that (1) holds for a generic family $\{ J_t\}$. Similarly, one can show that (1') holds for a generic family $\{ J_s\}$ as above.

In fact, we claim that there are no
$J_t$-holomorphic disks of negative Maslov index, somewhere injective or not. Indeed,
let $\{ J_t\}$ satisfy (1). By a result of Kwon-Oh
\cite{KO} (cf. \cite{Lazz2000, Lazz2011}), any non-parameterized $J_t$-holomorphic disk
in the class $C$
with boundary in $L$, viewed as a subset of $\C^2$, is a finite union of non-parameterized somewhere
injective $J_t$-holomorphic disks $\cD_1,\ldots, \cD_j$ and
$C= k_1[\cD_1]+\ldots+k_j [\cD_j]$, where
for each $i$ $[\cD_i]$ is the relative homology class of $\cD_i$ and $k_i\in\N$.
If $k_1\mu (\cD_1) +\ldots + k_j\mu (\cD_j) = \mu (C) <0$, then $\mu (\cD_i)<0$ for some
$i=1,\ldots,j$, in contradiction to the non-existence of somewhere injective $J_t$-holomorphic
disks of negative Maslov index, which proves the claim.

\begin{lemma}
\label{lem-lagr-tori-in-R4-no-non-const-maslov-0-disks}
For any (possibly not even regular) $J\in \cJ_t$, $0\leq t\leq T$, there
are no non-constant $J$-holomorphic Maslov-0 disks of area less than $\omega_t (A)$,
somewhere injective or not (recall that $A \in H_2(\C^2, L)$ is the class appearing in the hypothesis
of the theorem).
\end{lemma}

\bigskip
\noindent
{\bf Proof of Lemma~\ref{lem-lagr-tori-in-R4-no-non-const-maslov-0-disks}:}
Indeed, assume by contradiction that such a disk exists and denote its relative homology class by $C$. By the
hypothesis of the theorem, $\omega(B) \geq (1+k)\omega(A)$ and $(1+k){\sigma} \geq {\rho}$, so that
\[
\omega_t( B - k A) - \omega_t(A) = \omega(B) - (1+k) \omega(A) + ((1+k){\sigma} - {\rho})t \geq 0.
\]
Moreover, $\omega_t (C)>0$, since $C$ is non-constant and $J$ is compatible with $\omega_t$.
Thus
\begin{equation}
\label{eqn-omega-t-A-C}
\omega_t( B - k A) \geq \omega_t(A) > \omega_t (C) >0.
\end{equation}
On the other hand, since $\mu (C)=0$, by the hypothesis of the theorem, $C$ is an integral multiple
of $B - k A$, and therefore $\omega_t (C)$ has to be an integral multiple of $\omega_t( B - k A)$,
which is impossible by \eqref{eqn-omega-t-A-C}. This proves the lemma.\Qed

Continuing with the proof of the theorem,
we now show that for any almost complex structures $J_0\in \cJ_0$, $J_T \in \cJ_T$ that are
regular with respect to $p\in L$ and any
family $\{ J_t\}$, $0\leq t\leq T$, $J_t\in \cJ_t$, that connects
$J_0$ and $J_T$ and is regular with respect to $p$, the moduli space
$\cM_1 (A, \{ J_t\},p)$ is compact and hence, in particular,  so are $\cM_1 (A,J_0,p)$
and $\cM_1 (A,J_T,p)$.

It suffices to show that any sequence $\{ D_i \}$ in $\cM_1 (A, J_{t_i}, p)$ with
$\{ t_i\}\to s$ has a subsequence converging to an element of $\cM_1 (A, J_s, p)$. To prove this
claim note that, by Gromov compactness (see \cite{Frauenfelder}), {since $\pi_2 (\C^2)=0$}, the sequence $\{ D_i \}$ has
a subsequence converging to a bubbling configuration of non-constant non-parameterized
$J_s$-holomorphic disks $\cD_1,\ldots , \cD_l$ with boundary in $L$ whose homology classes
$[\cD_1],\ldots, [\cD_l]$ add up to $A$:
\[
[\cD_1]+\ldots + [\cD_l] = A.
\]
Thus, $0 < \omega_s ([D_i]) \leq \omega_s (A)$ for all $i=1,\ldots, l$ and therefore,
{by Lemma~\ref{lem-lagr-tori-in-R4-no-non-const-maslov-0-disks}},
none of the disks $\cD_i$ has Maslov index $0$. Moreover, since the family $\{ J_t \}$ is regular with respect to $p$, none
of the disks $\cD_i$ has a negative Maslov index. Since $\mu ([\cD_1])+\ldots + \mu ([\cD_l]) = \mu (A)=2$
and the Maslov indices of the $\cD_i$'s are all even, this means that $l=1$ -- that is, there is only one disk
in the bubbling configuration and its relative homology class is $A$.

Let us denote this (non-parameterized) disk by $D$. By the result of Kwon-Oh \cite{KO} (cf.
\cite{Lazz2000, Lazz2011}) mentioned above, the non-parameterized disk $D$, viewed as a subset of
$\C^2$, is a finite union of non-parameterized somewhere injective disks $D^{(1)},\ldots,D^{(m)}$
and $A$ is a linear combination with positive integral coefficients of the homology classes of
$D^{(1)},\ldots,D^{(m)}$. The $\omega_s$-areas of $D^{(1)},\ldots,D^{(m)}$ are all positive
numbers smaller than $\omega_s (A)$. Therefore, arguing as above, we get that all
$D^{(1)},\ldots,D^{(m)}$ must have positive even Maslov indices, meaning that $m=1$
and $D = D^{(1)} \in \cM_1(A,J_s, p)$.
Hence the sequence $\{ D_i \}$ has a subsequence converging to an element of $\cM_1 (A, J_s, p)$.
This finishes the proof that the smooth manifold
$\cM_1 (A, \{ J_t\}, p)$ is compact.

Thus the moduli space $\cM_1 (A, \{ J_t\}, p)$ is a compact smooth
1-dimensional cobordism between the compact 0-dimensional manifolds $\cM_1 (A,J_0,p)$ and
$\cM_1 (A,J_T,p)$.

In a similar way one can show that given a (regular) path $\gamma (s)$, $0\leq s\leq 1$, in $L$, $t\in [0,T]$
and a family $\{ J_s\}\subset \cJ_t$ regular with respect to $\gamma$, the moduli space
$\cM_1 (A, \{ J_s\}, \gamma)$ is a compact smooth
1-dimensional cobordism between the compact 0-dimensional manifolds $\cM_1 (A,J_0,\gamma (0))$ and
$\cM_1 (A,J_1,\gamma (1))$.

Let us summarize: for any $t\in [0,T]$ and any $J \in \cJ_t$ regular with respect to $p$ the moduli space $\cM_1 (A, J, p)$ is a compact
0-dimensional manifold. The number
$n_A (p,J): = \# \cM_1 (A,J,p)\mod 2$ -- that is, the mod-2 number of non-parameterized
$J$-holomorphic disks with one marked point that represent $A$ and pass through $p$ --
is independent of $p$ and $J$. Indeed, for a different $p'\in L$ and an almost complex structure $J'\in \cJ_t$
regular with respect to $p'$ the manifolds $\cM_1 (A, J, p)$ and $\cM_1 (A, J', p')$ are cobordant and therefore
$n_A (p,J) = n_A (p',J')$. This proves
part (A) of the theorem.

Moreover, for any $J_0 \in \cJ_0$ and  $J_T \in \cJ_T$ regular with respect to $p$ we have
$n_A (p,J_0) = n_A (p,J_T)$, since $\cM_1 (A,J_0,p)$ and $\cM_1 (A,J_T,p)$ are cobordant compact
$0$-dimensional manifolds.

In view of the above, if $n_A (p,J_0)$ is non-zero, then so is $n_A (p,J_T)$
for $J_0 \in \cJ_0$ and  $J_T \in \cJ_T$ regular with respect to $p$. In particular,
there exists a $J_T$-holomorphic disk in $\C^2$ with
boundary in $L$ in the relative homology class $A$ and therefore
\[
\omega_T (A) = \omega (A) - {\sigma}T >0,
\]
and hence
\[
T< \omega(A)/{\sigma}.
\]
Since this holds for any Lagrangian isotopy $\{ \psi_t\}_{0\leq t\leq T}$ as above, we get that
$\df_L (\alpha)$, which is the supremum of such $T$, is less or equal to $\omega (A)/{\sigma}$:
\[
\df_L (\alpha)\leq \omega (A)/{\sigma}.
\]
This proves part (B) of the theorem.
\Qed

\subsection{Proof of Theorem~\ref{thm-chekanov-tori-v2}}
\label{subsec-pf-thm-chekanov-tori-v2}

We calculate lower bounds for the $\bp$-invariants of
Chekanov tori $\Theta_a$. The proofs of both parts rely
on an increasing sequence of neighborhoods of $\Theta_a$ in $\C^2$.
We first explain this construction.

For $r \in \R_{> 0}$ let $D(r) \subset \C$ denote the standard open symplectic disk of
area $r$ and $T^*_r\SP^1 \subset T^*\SP^1$ be the subset of covectors
of norm $< r$ (here we choose the flat metric on $\SP^1 = \R / \Z$). Let $\lambda$
be the standard Liouville form on $T^*\SP^1$.

\begin{prop}\label{prop-nbhds-chekanov-torus}
For every $r > a$ there exists a neighborhood $U(r) \subset \C^2$ of
$\Theta_a$ such that $U(r)$ is symplectomorphic to $(D(r) \times
T^*_r\SP^1, dx \wedge dy + d\lambda)$ and $\Theta_a$ is identified with
$\partial D(a) \times \{ 0-section\}$.
\end{prop}

\begin{proof}
Let $\eta$ be the simple closed oriented curve in the open first quadrant $Q\subset \C$ used in the construction of the
Chekanov torus $\Theta_a$. The curve bounds a disk of area $a$ which is contained in
a larger open disk ${\cD}\subset Q$ of area $a+\delta$ for some small $\delta>0$.
The map
$\Xi : \SP^1 \times {\cD} \to \C^2$
given by
\[
\Xi(e^{2\pi i t}, z) = \frac{1}{\sqrt{2}}(e^{2 \pi it} z, e^{-2 \pi i t} z)
\]
is an embedding and its image contains $\Theta_a$ as well
as the Chekanov tori $\Theta_{a'}$ for every $0 < a' \leq a$. Denote the image
of $\Xi$ by $N$. We see that $\Xi$ preserves the symplectic
structure on ${\cD}$ and that $N$ is a coisotropic submanifold
with characteristic foliation generated by the $\SP^1$-action.
By the neighborhood theorem for coisotropic submanifolds
\cite{Gotay}, there exists a neighborhood $U$ of $N$ in $\C^2$
that is symplectomorphic to a neighborhood of the zero-section
in $E^*$, where $E \subset TN$ is the characteristic bundle
of $N$. One sees that the characteristic bundle of $N$ is
trivial, hence $E^* \simeq N \times \R \simeq {\cD} \times \SP^1
\times \R$. Since the {disk $\cD$ is} symplectic, the neighborhood
$U$ is symplectomorphic to a neighborhood $V$ of ${D (a+\delta)} \times
\{ 0-section\}$ in $D(a+\delta) \times T^*\SP^1$
with its standard split symplectic form $dx \wedge dy + d\lambda$.
By choosing a smaller neighborhood $U$ if necessary, we can
assume that $U$ is symplectomorphic to
$V = D (a+\delta) \times T^*_{\varepsilon} \SP^1$
for an $\varepsilon > 0$
so that $\Theta_a \subset U$ is identified with $\partial D(a) \times  \{ 0-section\}
\subset V$.

For $c > 0$ the map $\varphi_c: \C^2 \to \C^2$ given by
$(z_1, z_2) \mapsto \sqrt{c} (z_1, z_2)$  is a conformal
symplectomorphism. Recall that the neighborhood $U$ contains
all Chekanov tori $\Theta_{a'}$ for $0 < a' \leq a$. Now
let $0 < a' < a$ and $c = a / a'$,
then the image $\varphi_c(\Theta_{a'})$ is a Lagrangian
torus that is Hamiltonian isotopic to $\Theta_a$. The image
of the neighborhood, $\varphi_c(U)$, is then a neighborhood
of $\varphi_c(\Theta_{a'})$.
The neighborhood $\varphi_c(U)$ is symplectomorphic to
$V' = D(c(a+\delta)) \times T^*_{c\varepsilon} \SP^1$
(recall that $\SP^1 = \R / \Z$ and $\lambda = p \, d\theta$ is the standard Liouville form).

Let us now construct the wanted neighborhoods $U(r)$ for $r > a$. For
a given $r > a$ we choose $a' > 0$ so that for $c = a / a'$ we have
$c \, dx \wedge dy(D) > r$ and
$c \, \varepsilon > r$. Then $D(r) \times T^*_r \SP^1
\subset V'$ and this gives us a neighborhood $U(r)$.
Since $a' > 0$ can be chosen arbitrarily small, this
provides us with a neighborhood $U(r)$ for all $r > a$.
The statement follows from this.
\end{proof}

We return to the proof of part (A). For the lower bound
$\bp_{a} (m, n) \geq a/m$, following the dynamical
characterization of $\bp_{L}$ in
Section~\ref{subsec-defn-of-pbL-in-gen-case}, we use
a neighborhood that we obtain from
Proposition~\ref{prop-nbhds-chekanov-torus}
to construct for any $\varepsilon > 0$ a complete Hamiltonian
$H$ that has no chords from $X_0$ to $X_1$ of time-length
$< \tfrac{a}{m} - \varepsilon$.

For $m \geq 1$ and $n \in \Z$ fixed we choose $k \in
\N$ such that $k > a$ and $k > C := 4 a |n|$. Now
let $U(k)$ be a neighborhood of $\Theta_a$ as in
Proposition~\ref{prop-nbhds-chekanov-torus} and we
identify $U(k) \simeq D(k) \times T^*_k \SP^1$. Recall
that $\Theta_a$ is mapped to $\partial D(a) \times \{ 0-section\}$
under this identification.
Furthermore, we identify $\C \simeq \R^2$ and write the coordinates of
$D(k) \times T^*_k \SP^1$ as $(x,y, \theta, p)$.

First we define a partition of $\partial D(a)$.
Namely, partition $\SP^1$ into $4m$ closed arcs by setting
\[
\gamma_j := \left\{ e^{2\pi i \xi} \in \SP^1 \; \Big| \;
\xi \in \left[\frac{j - 1}{4m}, \frac{j}{4m} \right] \right\},
\;\; j = 1, \ldots, 4m.
\]
Using the identification $\partial D(a) \simeq \SP^1$ we
consider the $\gamma_j$ to be arcs of $\partial D(a)$. For $\theta
\in \R/\Z$ denote by $R(\theta): \R^2 \to \R^2$ the rotation by angle
$2\pi \theta$.
In $D(k) \times T^*_k \SP^1$ we define the
four sets,
\begin{equation}\label{eqn-admissible-quadruple}
\begin{aligned}
X_0 & = \bigcup\limits_{j \equiv 1\modd 4} \left\{ (R(n\theta)(x,y),
\theta, 0) \ | \  (x,y) \in \gamma_j, \ \theta \in \R/\Z \right\}, \\
X_1 & = \bigcup\limits_{j \equiv 3\modd 4} \left\{ (R(n\theta)(x,y),
\theta, 0) \ | \ (x,y) \in \gamma_j, \ \theta \in \R/\Z \right\}, \\
Y_0 & = \bigcup\limits_{j \equiv 0\modd 4} \left\{ (R(n\theta)(x,y),
\theta, 0) \ | \ (x,y) \in \gamma_j, \ \theta \in \R/\Z \right\}, \\
Y_1 & = \bigcup\limits_{j \equiv 2\modd 4} \left\{ (R(n\theta)(x,y),
\theta, 0) \ | \ (x,y) \in \gamma_j, \ \theta \in \R/\Z \right\},
\end{aligned}
\end{equation}
A brief calculation reveals that under the identification
$U(k) \simeq D(k) \times T^*_k \SP^1$ the quadruple
is mapped to an admissible quadruple associated to the fibration
$f_\alpha: \Theta_a \to \SP^1$, where $\alpha =
m\hGamma - n\hgamma \in H^1(\Theta_a)$.

Consider now the sector $\mathcal{S} \subset D(k)$ given by
\[
\mathcal{S} :=  D(k) \cap \left\{ r e^{2\pi i \xi} \; \Big|
\; r \geq 0, \ \xi \in \left[\frac{-1}{8m}, \frac{7}{8m}\right]
 \right\}.
\]
The intersection $\mathcal{S}\cap D(a)$ has area $a/m$ and
$\partial \mathcal{S}$ intersects $\partial D(a)$ in the arcs
$\gamma_4$ and $\gamma_{4m}$.
We consider a Hamiltonian $G: \mathcal{S} \to \R$ that
satisfies the following:
\begin{enumerate}
\item[(i)] $G$ has compact support in $\mathrm{Int}(\mathcal{S})$,
\item[(ii)] $G \equiv 0$ in a neighborhood of $(\gamma_{4m} \cup
\gamma_4) \cap \mathcal{S}$,
\item[(iii)] $G \equiv 1$ in a neighborhood of $\gamma_2$,
\item[(iv)] all chords of $G$ from $\gamma_1$ to $\gamma_3$ have
time-length $T > \tfrac{a}{m} - \varepsilon$,
\item[(v)] on $\mathcal{S}\cap D(a)$ the Euclidean norm of the gradient of $G$ is
bounded, $|(\partial_x G, \partial_y G)| \leq \tfrac{4 m}{\sqrt{a}}$.
\end{enumerate}

The existence of such a Hamiltonian follows easily. See
Figure~\ref{figure-pizza-slice-in-traj} for an illustration.
Namely, one can choose a piecewise-linear function that satisfies
all conditions except (iv) and then find a smooth approximation that
satisfies all conditions.

\begin{figure}[h]
    \centering
    \includegraphics[scale=0.75]{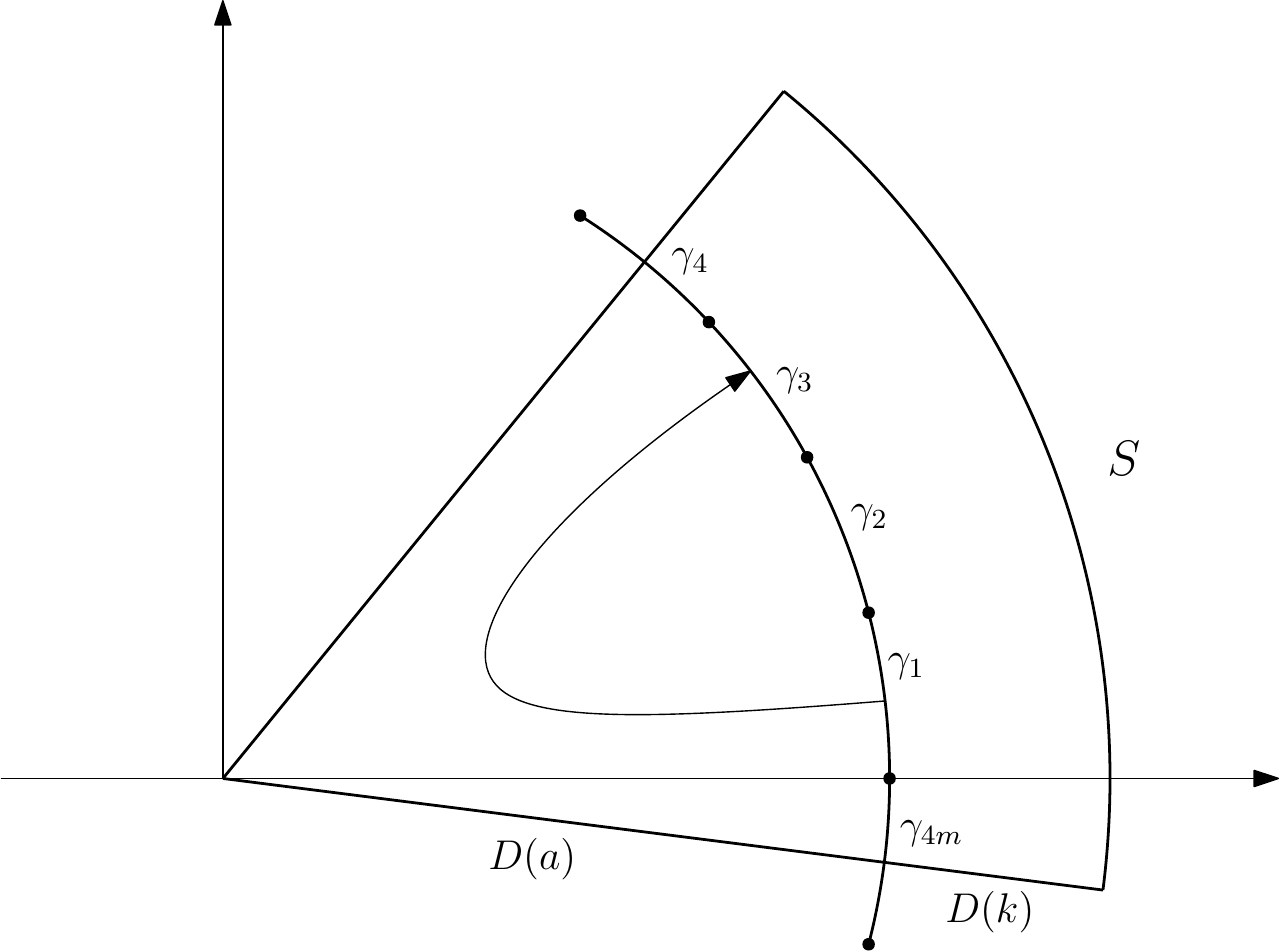}
    \caption{Hamiltonian chord of $G$ from $\gamma_1$ to $\gamma_3$
    in $\mathcal{S}$.}
    \label{figure-pizza-slice-in-traj}
\end{figure}

Extend $G$ to the whole disk $D(k)$ such that $G$ is periodic
under the rotation $R(1/m)$.
Since $G$ has support in $\mathrm{Int}(\mathcal{S})$,
this extension is smooth and has compact support. By abuse of
notation we denote the extension by $G$. Now choose a non-negative
function $\beta: T^*_k \SP^1 \to \R$ that is constant equal
to 1 for $|p| \leq C$ and has compact support in $T^*_k\SP^1$.

Define the Hamiltonian $H: D(k) \times T^*_k \SP^1 \to \R$ via
\[
H(x, y, \theta, p) = G(R(n\theta)(x,y))\beta(p).
\]
By construction $H$ is complete and satisfies $H \equiv 0$ on
$Y_0$ and $H \equiv 1$ on $Y_1$.
Since $\beta(p)$ is constant in $\{ |p| \leq C \}$, we see
that $\partial_p H \equiv 0$ in this region. This implies that
the Hamiltonian vector field of $H$ in $\{ |p| \leq C \}$ is
tangent to the fibers $\{ \theta = \mathrm{constant} \}$.
Thus all Hamiltonian chords of $H$ starting on
$\partial D(a) \times \{ 0-section \}$ and
contained in $\{ |p| \leq C \}$ will project
to the fibers.

Now for $\{ |p| \leq C \}$ and $(x,y) \in D(a)$ we estimate the
Euclidean norm of the differential,
\[
|\partial_{\theta} H| \leq |(\partial_x G, \partial_y G)| \cdot |n|
\cdot |(x,y)| \leq \frac{4m |n|}{\sqrt{\pi}}.
\]
This implies that all Hamiltonian chords
of $H$ starting in $\partial D(a) \times \{ 0-\mathrm{section} \}$
remain in the region $\{ |p| \leq C \}$ for times $t \in [0, a/m]$
by our choice of $C$. Therefore any chord of time-length $t \leq
a/m$ is contained in the fiber $\{ \theta = \mathrm{constant} \}$
and under the projection $D(k) \times T^*_k \SP^1 \to D(k)$ these
chords project to rotated chords of $G$. This proves that
any chord from $X_0$ to $X_1$ must have time-length
$T > a/m - \varepsilon$. This completes the proof of
part (A).

\medskip
\medskip
We prove part (B). We construct
a sequence of Hamiltonians that have chords from $X_0$ to
$X_1$ with time-length increasing to $+\infty$.

\begin{figure}[h]
    \centering
    \includegraphics[scale=0.75]{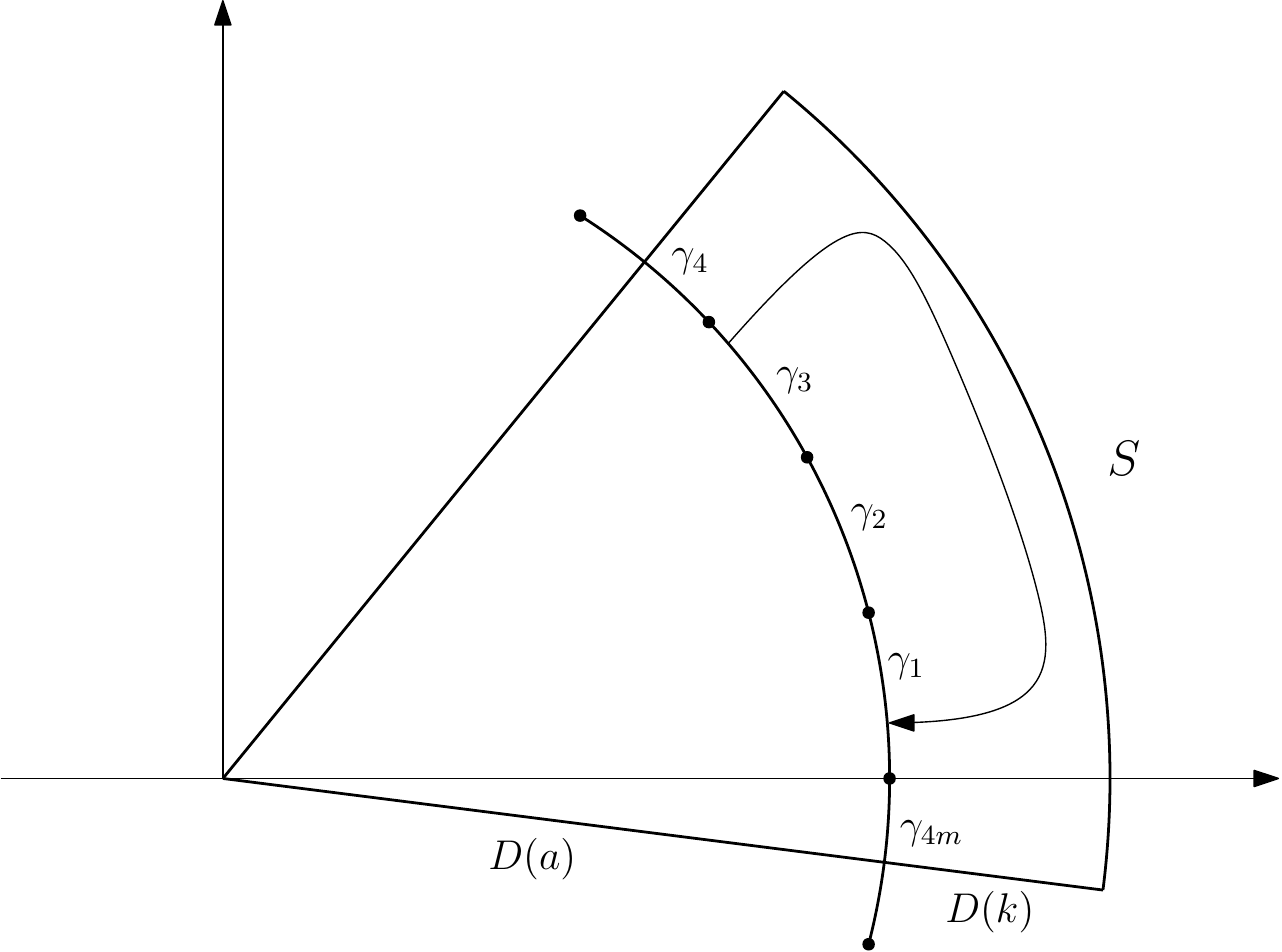}
    \caption{Hamiltonian chord of $G$ from $\gamma_3$ to $\gamma_1$
    in $\mathcal{S}$.}
    \label{figure-pizza-slice-out-traj}
\end{figure}

We consider the admissible quadruple associated to the
fibration $f_{\alpha}: \Theta_a \to \SP^1$, where $\alpha
= m\hGamma - n\hgamma$ with $m < 0$ and $n \in \Z$.
Following Remark~\ref{rem-bp-of-multiple-classes-defn}
we see that if $(X_0, X_1, Y_0, Y_1)$ is an admissible
quadruple for $\alpha$, then $(X_1, X_0, Y_0, Y_1)$ is
an admissible quadruple for $-\alpha$. Thus we may take
the admissible quadruple associated to $-m\hGamma + n\hgamma$
as constructed in \eqref{eqn-admissible-quadruple} and
interchange $X_0$ and $X_1$ to obtain an admissible
quadruple associated to $m\hGamma - n\hgamma$.

Following the construction in part (A), we find a
Hamiltonian $G: \mathcal{S} \to \R$ that satisfies
$G \equiv 0$ on $\gamma_{4m} \cup \gamma_4$ and
$G \equiv 1$ on $\gamma_2$ and consider its chords
from $\gamma_3$ to $\gamma_1$, see Figure~\ref{figure-pizza-slice-out-traj}.
In our sequence of symplectic neighborhoods we have $dx \wedge dy(
D(k)/D(a)) \to +\infty$. Hence we can choose for
all $k > a$ a sequence of Hamiltonians $G_k$ for which
the time-length $T$ of chords from $\gamma_3$ to $\gamma_1$
goes to $+\infty$. Repeating the construction as in
part (A) then gives us a sequence of Hamiltonians $H_k$
in $\C^2$ with the desired properties.
\Qed


\section{Lagrangian tori in $\C^n$ -- proofs}


\subsection{Basic properties of $\bp_\bx$, $\df_\bx$ for split Lagrangian tori in $\C^n$}
\label{sec-prop-basic-properties-of-pbL-for-split-tori-in-R2n}

\begin{prop}
\label{prop-basic-properties-of-pbL-for-split-tori-in-R2n} \

\medskip
\noindent
(A) {\bf Permutation invariance}: Let $\bx_\sigma\in\R^n$, $\alpha_\sigma
\in (\R^n)^*$ be the vectors obtained, respectively, from $\bx \in \R^n$,
$\alpha \in (\R^n)^*$ by a permutation $\sigma$ of the coordinates. Then
for any $\alpha$
\begin{equation}
\label{eqn-permutation-of-factors-in-split-torus}
\df_{\bx_\sigma} (\alpha_\sigma) = \df_{\bx} (\alpha), \ \
\bp_{\bx_\sigma} (\alpha_\sigma) = \bp_{\bx} (\alpha).
\end{equation}

\noindent
(B) {\bf Homogeneity in $\bx$}:
for any $c>0$
\[
\df_{c\bx} = \frac{1}{c} \df_{\bx},\ \
\bp_{c\bx} = \frac{1}{c} \bp_{\bx}.
\]

\noindent
(C) {\bf Semi-continuity with respect to $\bx$}:
\[
\{ \bx_i \} \to \bx\ \Longrightarrow\ \forall \, \alpha\in (\Z^n)^*:
\ \bp_{\bx} (\alpha) \leq \liminf\limits_{\bx_i} \bp_{\bx_i}(\alpha),
\]
\[
\{ \bx_i \} \to \bx\ \Longrightarrow\ \forall \, \alpha\in (\R^n)^*:
\ \df_{\bx} (\alpha) \leq \liminf\limits_{\bx_i} \df_{\bx_i} (\alpha).
\]

\noindent
(D) {\bf The product property}: For any $\alpha$
\[
\df_{(\bx, x_{n+1},\ldots, x_N)} (\alpha,0,\ldots,0)\geq \df_{\bx}
(\alpha).
\]
\[
\bp_{(\bx, x_{n+1},\ldots, x_N)} (\alpha,0,\ldots,0)\geq \bp_{\bx}
(\alpha).
\]
\end{prop}

\begin{proof}

Claim (A) holds, since any split torus obtained from $T^n (\bx)
= T^1 (x_1)\times\ldots\times T^1 (x_n)$ by a permutation of the
$T^1$-factors is Hamiltonian isotopic to $T^n (\bx)$.

Claim (B) follows easily from the fact that $(p,q)\mapsto \sqrt{c}
(p,q)$, $c>0$, is a conformal symplectomorphism.

Claim (C) follows from the semi-continuity property of $\bp_L$
(see \eqref{eqn-Hausdorff-convergence}) and $\df_L$ (see \eqref{eqn-def-semi-cont}).

Claim (D) follows from the product property of $pb_4^+$ (see
\eqref{eqn-pb-products}).
\end{proof}


\subsection{Proof of Theorem~\ref{thm-lagr-tori-in-Cn}}
\label{sec-lagr-tori-in-Cn-proofs}

We now prove that for certain Lagrangian tori $L \subset \C^n$
and specific cohomology classes $\alpha \in H^1(L; \R)$ there
are upper bounds for the associated Lagrangian isotopies. This
proof follows the same route as in the proof of
Theorem~\ref{thm-lagrangian-tori-in-R4}. We indicate the changes
here and refer the reader to the proof of
Theorem~\ref{thm-lagrangian-tori-in-R4} for more details.

Consider a Lagrangian isotopy
\[
\psi = \{ \psi_t: L \to \C^n \}, \ 0 \leq t \leq T, \
\psi_0 = \iota,
\]
such that
\[
[\omega]^{\psi}_t = [\omega]_L - t \partial\alpha,
\]
with $L \subset \C^n$ and $\alpha$ as in the statement
of the theorem.

For simplicity we consider the dual picture. Namely, for
$0 \leq t \leq T$ there exists a family of compactly supported
diffeomorphisms $\varphi_t: \C^n \to \C^n$, $\varphi_0 = \mathrm{Id}$,
such that $\varphi_t(L) = \psi_t(L)$. By pulling back our symplectic
form $\varphi_t^*\omega =: \omega_t$ we may consider a fixed Lagrangian
$L \subset (\C^n, \omega_t)$. We have $H_2(\C^n, L) \simeq
\Z\langle A_1,\ldots, A_n \rangle$ and $\omega_t(A_1) =
a - \sigma t$ and $\omega_t(A_i) = b - \rho t$
for $i = 2, \ldots, n$.

Let $\cJ_t$, {$t\in [0,T]$}, be the space of almost complex structures on $\C^n$ compatible
with the symplectic form $\omega_t$ and let $C \in H_2(\C^n, L)$ be a relative
homology class. Choose a point $p \in L$. We define the moduli spaces
$\cM_1(C, J)$, $\cM_1(C, J, p)$ and $\cM_1(C, \{J_t \}, p)$ as in the proof of
Theorem~\ref{thm-lagrangian-tori-in-R4}. An almost complex structure $J \in \cJ_t$
is called {\it regular}, if $\cM_1(C, J)$ is a (transversally cut out) smooth manifold of dimension
\[
\dim \cM_1(C, J) = \dim L + \mu(C) - 2,
\]
and {\it regular with respect to $p$}, if
the space $\cM_1(C, J, p)$ is a (transversally cut out) smooth
manifold of dimension
\[
\dim \cM_1(C, J, p) = \mu(C) - 2.
\]

Further on we will assume $n = \dim L$ is even -- the arguments for $n$ odd are similar.

Let $S\in (0,T]$. We will say that a family $\{ J_t \}$, $0 \leq t \leq S$, $J_t \in \cJ_t$, is
{\it regular with respect to $p$} if
\begin{enumerate}
\item[(1)] for any $t \in [0, S]$ the spaces $\cM_1(C, J_t)$ are empty for all $C$
with $\mu(C) < 2 - n$,
\item[(2)] $\cM_1(C, \{J_t \}, p)$ is a (transversally cut out) smooth manifold of dimension
\[
\dim \cM_1(C, \{J_t \}, p) = \dim \cM_1(C, J, p) + 1 = \mu(C) - 1
\]
with boundary $\cM_1(C, J_0, p) \cup \cM_1(C, J_S, p)$.
\end{enumerate}

Again, by the same standard regularity and transversality arguments, used in the proof of
Theorem~\ref{thm-lagrangian-tori-in-R4}, that for any $J_0 \in \cJ_0$ and
${J_S} \in {\cJ_S}$, that are regular with respect to $p$, a generic family $\{ J_t \}$,
$0 \leq t \leq S$, $J_t \in \cJ_t$, connecting $J_0$ and $J_S$ satisfies (2).
In order to show that condition (1) also holds for a generic family $\{ J_t \}$ note
that since $L$ is orientable,
the Maslov index of any disk is even and therefore, since $\dim L$ is even,
$\dim \cM_1(C, J)$ is even. Thus, if $\mu(C) < 2 - n$, then
$\mathrm{virtual} \dim \cM_1(C, J) \leq -2$ for any $J \in \cJ_t$ and $t \in [0,{S}]$, meaning that
the existence of \emph{somewhere injective} $J_t$-holomorphic disks of such
Maslov indices is a codimension-2 phenomenon and can be avoided by a
generic choice of $\{ J_t \}$.

We remark the following difference to the proof of Theorem~\ref{thm-lagrangian-tori-in-R4}.
For a general $C \in H_2(\C^n, L)$ with $\mu(C) < 2-n$, we cannot a priori exclude
the existence of {\it not somewhere injective} $J_t$-holomorphic disks in the class $C$.
However, we will show now that {\it for $t \in [0, a/\sigma)$} all $J_t$-holomorphic disks
of Maslov index lying in $[2-n,0]$, somewhere injective or not, can be excluded using our
assumptions on $A_1,\ldots, A_n$ and $\alpha$.

Set $a(t) := a - \sigma t$ and $b(t) := b - \rho t$.

\begin{lemma}
\label{lem-lagr-tori-in-Cn-no-Maslov-minus-2l-disks-of-area-smaller-than-a-t}
For $t \in [0, a/\sigma)$ there is no
$J_t$-holomorphic
disk with boundary in $L$
with Maslov index equal to $-2l$ for $l \in \{ 0, 1, \ldots, (n-2)/2 \}$ that has positive
symplectic area smaller than $\omega_t(A_1) = a(t)$.
\end{lemma}

\bigskip
\noindent
{\bf Proof of Lemma~\ref{lem-lagr-tori-in-Cn-no-Maslov-minus-2l-disks-of-area-smaller-than-a-t}:}
Let $t \in [0, a/\sigma)$.
For $l \in \{ 0, 1, \ldots, (n-2)/2 \}$ let
$D \in H_2(\C^n, L)$ be a class such that $\mu(D) = -2l$ and
$D$ has positive symplectic area.

A brief calculation shows that $b(t) - a(t) > 0$ and $a(t)>0$, since $t \in [0, a/\sigma)$ and
$\rho \leq \tfrac{\sigma (n+2)}{2} \leq \tfrac{b \sigma}{a}$. Hence, we also have that $b(t)>0$.

A calculation of the symplectic area of $D$
reveals that, since $b(t) > 0$, $b(t)-a(t)>0$ and $\omega_t(D) >0$, for each $t \in [0, a/\sigma)$ we have
\[
\omega_t(D) = k(b(t) - a(t)) - lb(t)
\]
for some $k \in \Z$ such that
\[
k > \frac{lb(t)}{b(t) - a(t)} > l,
\]
where the last inequality actually means that
\[
k\geq l+1,
\]
because $k,l\in\Z$.

Hence,
\begin{equation}\label{eq-ineq-sympl-area-maslov-2l-disks}
\omega_t(D)\geq (l+1)(b(t) - a(t)) - lb(t) = b(t) - (l+1)a(t) \geq a(t).
\end{equation}
Here the last inequality in \eqref{eq-ineq-sympl-area-maslov-2l-disks} can be deduced from the assumptions $l \in \{ 0, 1, \ldots, (n-2)/2 \}$,
$\rho/\sigma\leq (n+2) / 2\leq b/a$.
Inequality \eqref{eq-ineq-sympl-area-maslov-2l-disks}
implies the lemma. \Qed

Continuing with the proof of the theorem, assume $0<S < a/\sigma$.
We will now show that for any almost complex structures
$J_0 \in \cJ_0$ and $J_S \in \cJ_s$ regular with respect to $p$
and any family $\{J_t\}$,
$0 \leq t \leq S$, $J_t \in \cJ_t$, that connects $J_0$ and $J_S$ and is regular with respect to $p$,
the moduli space $\cM_1(A_1, \{J_t\}, p)$
is compact and hence, in particular, the moduli spaces $\cM_1(A_1, J_0, p)$
and $\cM_1(A_1, J_S, p)$ are compact.

As in the proof of Theorem~\ref{thm-lagrangian-tori-in-R4} it suffices
to consider a sequence $\{ D_i \}$ in $\cM_1(A_1, J_{t_i}, p)$ with
$\{ t_i \} \to s\in [0,S]$
{and show that it has a subsequence converging to an element of $\cM_1(A_1, J_s, p)$}.
By Gromov compactness (see \cite{Gromov}), since $\pi_2 (\C^n)=0$, the sequence $\{ D_i \}$
has a subsequence converging to a bubbling configuration of non-constant
non-parameterized $J_s$-holomorphic disks $\cD_1, \ldots, \cD_l$ with
boundary in $L$ whose homology classes $[\cD_1], \ldots, [\cD_l]$
add up to $A_1$:
\[
[\cD_1] + \ldots + [\cD_l] = A_1.
\]
By the result of Kwon-Oh \cite{KO}, each disk $\cD_i$, viewed as a subset
of $\C^n$, is a finite union of non-parameterized somewhere injective
$J_s$-holomorphic disks $\cD_i^{(1)}, \ldots, \cD_i^{(i_n)}$ so that
\[
\mu(\cD_i) = k_1 \mu(\cD_i^{(1)}) + \ldots + k_{i_n} \mu(\cD_i^{(i_n)})
\]
where the coefficients $k_i$ are positive integers.
Now $0 < \omega_s([\cD_i]) \leq \omega_s(A_1)$ for all $i = 1, \ldots, l$ and
this implies that $0 < \omega_s([\cD_i^{(j)}]) \leq \omega_s(A_1)$ for all
somewhere injective disks.
By {Lemma~\ref{lem-lagr-tori-in-Cn-no-Maslov-minus-2l-disks-of-area-smaller-than-a-t}},
this means that $\mu(\cD_i^{(j)}) \notin \{ 2 -n, \ldots, -2, 0\}$.
Since the family $\{ J_t \}$ is regular with respect to $p$, somewhere
injective $J_s$-holomorphic
disks of Maslov index $< 2 - n$ do not exist. Putting everything
together we conclude that
for all $i,j$ $\mu(\cD_i^{(j)}) >0$,
which, in fact, means that
 $\mu(\cD_i^{(j)}) \leq 2$ (recall that $L$ is orientable and therefore the Maslov indices of disks with boundary in $L$ are all even). Therefore
$A_1 = [\cD_1]$, the disk $\cD_1$ lies
in $\cM_1(A_1, J_s, p)$ and the sequence $\{ D_i \}$ has a subsequence converging to an element of
$\cM_1(A_1, J_s, p)$. This finishes the proof that the smooth
manifold $\cM_1(A_1, \{J_t\}, p)$ is compact.

Thus for any $J_0\in\cJ_0$, $J_S\in\cJ_S$ that are regular with respect to $p$ and a family
$\{ J_t \}_{0\leq t\leq S}$ that connects $J_0$ and $J_S$ and is regular with respect to $p$
the moduli space $\cM_1(A_1, \{J_t\}, p)$ is a compact
smooth 1-dimensional cobordism between the compact 0-dimensional
manifolds $\cM_1(A_1, J_0, p)$ and $\cM_1(A_1, J_S, p)$, as long as $0<S<a/\sigma$.
(Note the difference with the proof of Theorem~\ref{thm-lagrangian-tori-in-R4} where a similar
claim was proved for {\it any} $S$ such that $\omega_t$ is symplectic for all $t\in [0,S]$).
This implies that $n_{A_1}(J_0, p) = n_{A_1}(J_S, p)$.

In a similar way one can show that given a (regular) path $\gamma (s)$, $0\leq s\leq 1$, in $L$,
$t\in [0,S]$
and a family $\{ J_s\}\subset \cJ_t$ regular with respect to $\gamma$, the moduli space
$\cM_1 (A, \{ J_s\}, \gamma)$ is a compact smooth
1-dimensional cobordism between the compact 0-dimensional manifolds $\cM_1 (A,J_0,\gamma (0))$ and
$\cM_1 (A,J_1,\gamma (1))$, and therefore $n_{A_1}(J_0, \gamma (0)) = n_{A_1}(J_1, \gamma (1))$. The latter claim proves part (A)
of the theorem for $n$ even. The case for $n$ odd is similar.

In order to prove part (B) we need to show that $T\leq a/\sigma$. Let us assume by contradiction that
$a/\sigma<T$.
As we have shown above, $n_{A_1}(J_0, p) = n_{A_1}(J_S, p)$ for any $S\in (0,a/\sigma)$.
Thus if $n_{A_1}(J_0, p)$ is non-zero, then so is $n_{A_1}(J_S, p)$.
This implies that for every $S < a/\sigma$ there exists a $J_S$-holomorphic
disk in $\C^n$ with boundary on $L$ in the relative homology class $A_1$.
Passing to the limit $S\to a/\sigma$ and applying again the Gromov compactness, we see that there must exist a configuration of $J_{a/\sigma}$-holomorphic disks whose total homology class is $A_1$ (and hence not all disks in the configuration are constant) and whose total $\omega_{a/\sigma}$-area is $\omega_{a/\sigma} (A_1) =0$, which is impossible since the area of each non-constant $J_{a/\sigma}$-holomorphic disk has to be positive.
Hence, we obtain a contradiction and this proves (B).
\Qed

\begin{rem}
{\rm We remark that the proof of Theorem~\ref{thm-lagr-tori-in-Cn}
does not generalize to the settings where $\omega (A_i)\neq \omega(A_j)$ for some $2\leq i,j\leq n$, $i\neq j$.
The reason
for this comes from the fact that if in the Lagrangian isotopy
there exist two Maslov-0 disks with positive symplectic areas that
are rationally independent (e.g. $A_2 - A_1$ and $A_3 - A_2$),
then one can always find a Maslov-0 disk with positive symplectic area
arbitrarily close to 0. Hence one can not exclude Maslov-0 disks
from any bubbling configuration.}
\end{rem}


\subsection{Proof of Theorem~\ref{thm-computations-for-split-tori-in-Cn}}
\label{sec-pf-thm-computations-for-split-tori-in-Cn}

\noindent
{\bf Proof of part (A).} The tori $T^n (\bx)$ are exactly the regular
orbits of the standard Hamiltonian $\T^n$-action on $\C^n$ and the
regular level sets of its moment map $\Phi :\C^n\to (\R^n)^*$ whose image is
the non-negative quadrant $\Delta\subset(\R^n)^*$.

For $\alpha=(m_1, \ldots, m_n)$ one readily sees that
$\cI(\bx,\alpha)$, which is the open part of the intersection of the
ray $\bx - t \alpha$, $0<t<+\infty$, and $\Delta$, is an infinite ray
if and only if all $m_i$ are non-positive. Otherwise, $\cI(\bx,\alpha)$
is an interval and its rational length is given by $\min_{i, m_i>0}\ x_i/
m_i$. By Theorem~\ref{thm-toric-case-low-bound}, this proves part
(A) of the theorem. \Qed

\bigskip
\noindent {\bf Proof of part (B).}
The upper bound $\df_\bx (k,\ldots,k)\leq x/k$ for $k \in \N$ follows from
Theorem~\ref{thm-partial-alpha-proportional-to-omega-upp-bound-on-bpL-alpha},
since $\C^n$ does not admit weakly exact closed Lagrangian submanifolds
by a famous result of Gromov \cite{Gromov}.

The lower bound $\bp_\bx (k,\ldots,k)\geq x/k$ follows immediately from
part (A) of the theorem. \Qed

\bigskip
\noindent {\bf Proof of part (C).}
Let $p_1,\ldots,p_n,q_1,\ldots,q_n$ be the standard Darboux coordinates
on $\R^{2n}=\C^n$ so that $z_j=p_j+i q_j$, $j=1,\ldots,n$.

According to the assumption, $x_i/k > x_{min}$, that is, $x_i/k > x_j$ for
some $j \neq i$. By \eqref{eqn-permutation-of-factors-in-split-torus},
we may assume without loss of generality, that $i=n, j=1$, that is,
$x_n/k > x_1$. Let us show that $\bp_{\bx} (k e_n) = +\infty$.

Observe that a circle bounding a round disk of a certain area in $\R^2$
can be mapped by an area-preserving map arbitrarily close to the boundary
of a square of the same area. Together with the semi-continuity, product
and symplectic invariance properties of $pb^+_4$ (see
Section~\ref{subsec-pb4-generalities}) this easily implies that it is enough
to prove
\begin{equation}
\label{eqn-pb4+-zero-for-four-parts-of-a-thin-torus}
pb_4^+ (X_0,X_1,Y_0,Y_1) = 0,
\end{equation}
where the admissible quadruple $X_0,X_1,Y_0,Y_1$ is defined as
follows: for $i = 1, \ldots, n-1$
denote by $\Pi_i$ the boundary of the square $[0, \sqrt{x_i}]
\times [0, \sqrt{x_i}] \subset (\R^2(p_i, q_i), dp_i \wedge dq_i)$.
For $i = n$ consider the rectangle $[0, x_n] \times [0, 1]
\subset (\R^2(p_n, q_n), dp_n \wedge dq_n)$. Now choose
$\varepsilon > 0$ such that
\begin{equation}\label{eq-ineq-pb4-nonprim-classes-split-tori}
\frac{x_n}{k} - 4 \varepsilon > x_1.
\end{equation}
As in the proof of Theorem~\ref{thm-pbL-surfaces} we choose
a partition of $[0, x_n] \times \{ 1 \}$ into $4k - 3$ intervals
$\gamma_1, \ldots, \gamma_{4k - 3}$, ordered from right to left,
such that for $i \equiv 0, 2, 3 \, \modd 4$ the intervals $\gamma_i$
have length $\varepsilon$, the interval $\gamma_1$ has length
$\tfrac{x_n}{k}$ and the remaining intervals have length
$\tfrac{x_n}{k} - 3\varepsilon$. Note that if $k = 1$ then
$\gamma_1 = [0, x_n] \times \{ 1 \}$. We define
\[
\hX_0 = [0, x_n]\times \{ 0 \} \cup \bigcup\limits_{i \equiv 3 \modd 4} \gamma_i,
\ \ \hX_1 = \bigcup\limits_{i \equiv 1 \modd 4} \gamma_i,
\]
\[
\hY_0 = [0,1] \times \{ 0 \} \cup \bigcup\limits_{i \equiv 2 \modd 4} \gamma_i,
\ \ \hY_1 = [0,1] \times \{ x_n \} \cup \bigcup\limits_{i \equiv 0 \modd 4} \gamma_i.
\]
We then set
\begin{align*}
X_0 & := \Pi_1 \times \ldots \times \Pi_{n-1} \times \hX_0, \\
X_1 & := \Pi_1 \times \ldots \times \Pi_{n-1} \times \hX_1, \\
Y_0 & := \Pi_1 \times \ldots \times \Pi_{n-1} \times \hY_0, \\
Y_1 & := \Pi_1 \times \ldots \times \Pi_{n-1} \times \hY_1,
\end{align*}
which gives us our admissible quadruple.

Now choose a constant $C$ that satisfies
\begin{equation}\label{eq-ineq-pb4-nonprim-classes-split-tori-nr2}
\left( \frac{x_n}{k} - 4 \varepsilon \right) \frac{1}{\sqrt{x_1}}
> C > \sqrt{x_1}.
\end{equation}
Such a constant exists by inequality~\eqref{eq-ineq-pb4-nonprim-classes-split-tori}.
On $\R^2(p_n, q_n)$ we define a piecewise-linear function
$G_{\varepsilon}: \R^2 \to \R$ that satisfies:
\begin{enumerate}
\item[(i)] $G_{\varepsilon} \equiv 0$ on $\widehat{Y}_0$ and $G_{\varepsilon} \equiv C$
on $\widehat{Y}_1$,
\item[(ii)] $G_{\varepsilon}$ only depends on $p_n$.
\end{enumerate}
One can choose $G_{\varepsilon}$ to satisfy $\partial_x G_{\varepsilon}
\leq C / (\tfrac{x_n}{k} - 3 \varepsilon)$ on $\pi(\hX_1) \times \R$, where
$\pi(p_n, q_n) = p_n$ is the projection. Therefore
there exists a smooth approximation $G: \R^2 \to \R$ that is complete,
satisfies (i) and (ii) and has slope $\partial_x G < C / (\tfrac{x_n}{k}
- 4 \varepsilon)$ on $\pi(\hX_1) \times \R$. This implies that all chords
of $G$ from $\widehat{X}_0$ to $\widehat{X}_1$ have time-length
$T > (\tfrac{x_n}{k} - 4\varepsilon)/ C$.

Now consider the Hamiltonian $H: \R^{2n} \to \R$ defined by
\[
H(p, q) := p_1 + G(p_n, q_n).
\]
$H$ is complete and satisfies
\[
\min_{Y_1}\ H = C, \ \ \max_{Y_0}\ H = \sqrt{x_1},
\]
and thus, by inequality~\eqref{eq-ineq-pb4-nonprim-classes-split-tori-nr2},
\[
\min_{Y_1}\ H -  \max_{Y_0}\ H > 0.
\]
Again by inequality~\eqref{eq-ineq-pb4-nonprim-classes-split-tori-nr2}
we have
\[
\left( \frac{x_n}{k} - 4 \varepsilon \right) \frac{1}{C}
> \sqrt{x_1}.
\]
Under the projection $\R^{2n} \to \R^2(p_n, q_n)$ the chords of $H$
map to the chords of $G$. All chords of $G$ from $\hX_0$ to $\hX_1$
have time-length $T > (\tfrac{x_n}{k} - 4\varepsilon)/ C$. Now under
the projection $\R^{2n} \to \R^2(p_1, q_1)$ the chords of $H$ map
to vertical lines. We see that for any $T > \sqrt{x_1}$ the image of
$\Pi_1 \subset \R^2(p_1, q_1)$ under the Hamiltonian flow at time $T$
does not intersect $\Pi_1$. Thus there are no chords of $H$ from
$X_0$ to $X_1$, which, by the dynamical characterization
of $pb_4^+$, proves \eqref{eqn-pb4+-zero-for-four-parts-of-a-thin-torus} and part (C).
\Qed


\section{Lagrangian tori in $\C P^n$ and $\SP^2 \times \SP^2$ -- proofs}

We now prove that for Lagrangian tori $L$ in $\C P^n$ or
$\SP^2 \times \SP^2$ and certain cohomology classes
$\alpha \in H^1(L; \R)$ there are upper bounds for
$\df_L(\alpha)$.


\subsection{Proof of Theorem~\ref{thm-tori-in-CPn}}
\label{sec-pf-thm-cpn-toric-fiber-upp-bound}

Consider a Lagrangian isotopy
\[
\psi=\{ \psi_t : L \to \C P^n\}, \ 0\leq t\leq T, \ \psi_0=\iota,
\]
such that
\[
[\omega]^\psi_t = [\omega]_{L} - t \partial\alpha.
\]
By contradiction assume $T > C$ and consider $L_C := \psi_C(L)$.
By the assumption of the theorem, we have $[\omega]^\psi_C \in
H^2(\C P^n, L; \frac{1}{n} \Z)$ and
therefore $\omega(H^2(\C P^n, L_C)) \subset \frac{1}{n}\Z$. Since
$L$ is a torus the group $\omega(H^2(\C P^n, L_C))$ has $n + 1$
generators $y_1, \ldots, y_n, 1$. Denote by $c\in [0,+\infty) \cap
\frac{1}{n}\Z$ its positive generator. However, by a theorem of
Cieliebak and Mohnke \cite{Ciel-Mohnke},
\[
c\leq \frac{1}{n+1}.
\]
We obtain a contradiction and therefore $T < C$.
\Qed


\subsection{Proof of Theorem~\ref{thm-upp-bound-bpL-S2timesS2}}
\label{sec-pf-thm-S2-times-S2-upp-bound}

Consider a Lagrangian isotopy
\[
\psi=\{ \psi_t : L \to \SP^2 \times \SP^2\}, \ 0\leq t\leq T, \ \psi_0=\iota,
\]
such that
\[
[\omega]^\psi_t = [\omega]_{L} - t \partial\alpha.
\]
By contradiction, assume $T > C$ and consider $L_C := \psi_C(L)$.
By the assumption of the theorem we have
$[\omega]^\psi_C \in H^2(\SP^2 \times \SP^2, L; \Z)$ and
therefore $\omega(H^2(\SP^2 \times \SP^2, L_C)) \subset \Z$. Therefore
the group $\omega(H^2(\SP^2 \times \SP^2, L_C))$
is generated by $1$. However, by a theorem of
Dimitroglou Rizell, Goodmann and Ivrii (\cite{DRizell-Goodman-Ivrii},
Proposition 5.3), for any torus $L \subset \SP^2 \times \SP^2$
the positive generator $c$ of $\omega(H^2(\SP^2 \times \SP^2, L))$ satisfies
\[
c\leq \frac{1}{2}.
\]
We obtain a contradiction and therefore $T < C$.
\Qed

\bigskip
\noindent
{\bf Acknowledgments.}
We are very grateful to Leonid Polterovich for enlightening discussions,
stimulating questions and crucial comments and corrections. We thank
Emmanuel Opshtein for useful discussions and for suggesting to us the
original idea of the proof of Theorem~\ref{thm-lagrangian-tori-in-R4}.
We thank Paul Biran, Georgios Dimitroglou Rizell and Yakov Eliashberg for
useful conversations and questions related to our work.
Finally, we thank the anonymous referee for pointing out
mistakes in the first version of the paper.

\bibliographystyle{alpha}

\bigskip
\medskip

\noindent Michael Entov \\
Department of Mathematics \\
Technion - Israel Institute of Technology \\
Haifa 32000, Israel \\
\texttt{entov@math.technion.ac.il} \\
\medskip

\noindent Yaniv Ganor \\
School of Mathematical Sciences \\
Tel Aviv University \\
Tel Aviv 69978, Israel \\
\texttt{yanivgan@post.tau.ac.il} \\
\medskip

\noindent Cedric Membrez \\
School of Mathematical Sciences \\
Tel Aviv University \\
Tel Aviv 69978, Israel \\
\texttt{cmembrez@post.tau.ac.il}  \\
\medskip

\end{document}